\documentclass[12pt]{article}

\pdfoutput=1
\usepackage{mathtools}
\usepackage{amssymb}
\usepackage{graphicx}
\usepackage{subcaption}
\usepackage{enumerate}
\usepackage{amsthm}
\usepackage{amsmath}
\usepackage{cite}
\usepackage{hyperref}
\usepackage[titletoc]{appendix}
\usepackage[nottoc,notlot,notlof]{tocbibind}
\usepackage{tikz}
\usepackage{pgfplots}
\pgfplotsset{%
   every tick label/.append style = {font=\tiny},
   every axis label/.append style = {font=\scriptsize}
}

\allowdisplaybreaks[4]

\begin{document}

\begin{titlepage}

    \includegraphics[scale=0.2]{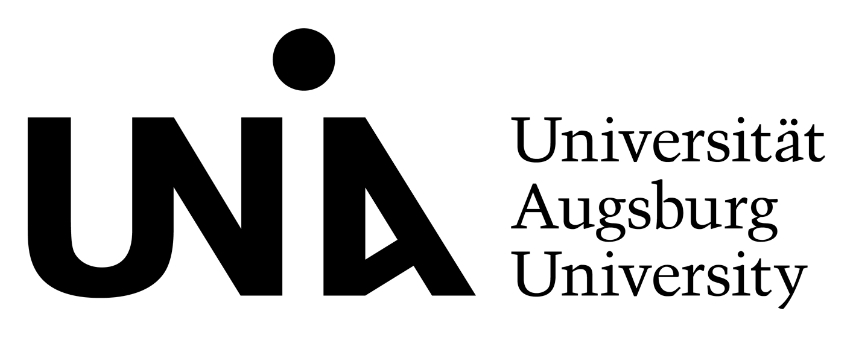}

 \begin{center} \large 
    \vspace*{0.5cm}
    Masterarbeit
    \vspace*{0.8cm}

    {\huge The curve shrinking flow, compactness and its relation to scale manifolds}
    \vspace*{1.2cm}
    
    \includegraphics[width=\linewidth]{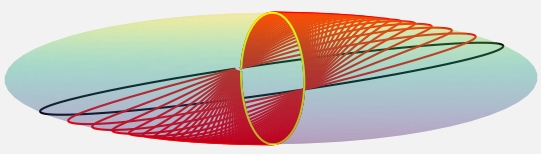}
    \vspace*{1.2cm}
    
    Oliver Neumeister
    \vspace*{1cm}

    20.07.2020
    \vspace*{2cm}

    Betreuung: Prof. Dr. Urs Frauenfelder \\[0.25cm]
    Institut f\"ur Mathematik \\[0.25cm]
    Universit\"at Augsburg \\[0.25cm]
  \end{center}
\end{titlepage}

\newpage
\thispagestyle{empty}
\mbox{}

\tableofcontents

\theoremstyle{plain}
\newtheorem{thm}{Theorem}[section]
\newtheorem{prop}[thm]{Proposition}
\newtheorem{claim}[thm]{Claim}
\newtheorem{lemma}[thm]{Lemma}
\newtheorem{cor}[thm]{Corollary}
\newtheorem{defn}[thm]{Definition}  

\theoremstyle{definition} 
\newtheorem{rmk}[thm]{Remark}
\newtheorem{exa}[thm]{Example}

\numberwithin{equation}{section}

\newpage

\section{Introduction}

In the Morse theory on compact manifolds, sequences of gradient flow lines between critical points always have convergent subsequences. Such a subsequence converges to another gradient flow line, potentially a broken one, between the same critical points.

This master thesis looks at the length functional on embedded loops and how its gradient flow lines exhibit the same behaviour as the gradient flow lines in Morse theory. The loop space will be viewed as an sc-manifold, in order to have the length as a smooth functional.

\vspace*{0.5cm}

In the first section the functional analytical foundation for the space of gradient flow lines will be built. For the space of gradient flow lines to be an sc-Banach space, we need to make the space of sc-smooth maps from $\mathbb{R}$ into an sc-Banach space to become an sc-Banach space.
To achieve this we use Bochner-Sobolev spaces $W^{n,\hat{p}}(I,\hat{X})$, that is spaces of weakly differentiable maps using the Bochner integral, where $\hat{X}$ is a n+1-tuple of Banach spaces.

The first important step is to get a compact embedding between these Bochner-Sobolev spaces. This is achieved via an extended version of the Aubin-Lions lemma. The Aubin-Lions lemma for $n = 1$ is commonly used in the theory of nonlinear partial differential equations, and thus a lecture on this topic by Prof. Schmidt has served as inspiration for this part. As the Aubin-Lions lemma only works for compact $I \subset \mathbb{R}$, we will extend it through the use of weight functions.

For the second step of proving, that smooth functions with compact support form a dense subset in all these Bochner-Sobolev spaces, we will use a version of the Meyers-Serrin theorem.

\vspace*{0.5cm}

In the second section we will look at the length functional. As the length on the loop space over a manifold has infinite dimensional critical manifolds, since the length does not change under reparametrization, we will only look at embedded loops. There it will be possible to quotient out the sc-smooth reparametrization action. We are going to get a space of unparametrized embedded loops E(M), that is an sc-manifold. On E(M), at least for most Riemannian manifolds M, the length functional is sc-smooth and has isolated critical points.

\vspace*{0.5cm}

In the third and last section of this thesis we will first look at the flow by curvature on surfaces. This curvature flow is the same as the flow prescribed by the gradient of the length functional. We will use this interchangeability of curvature and gradient flow to prove the convergence of the gradient flow lines. Lastly we will observe how the gradient flow lines of the length break. This behavior is very similar to Morse theory and thus this sections is partly influenced by a lecture on Morse theory by Prof. Frauenfelder.

\vspace*{0.5cm}

In the appendix there are several images of the curvature flow displayed. These are a result of a programming project accompanying this master thesis. A few of these images are also used in examples in the second and third section.

\vspace*{0.5cm}

My thanks goes to Urs Frauenfelder for proposing this topic to me and for discussing the directions and ideas of this thesis.

\newpage

\section{Bochner-Sobolev spaces} \label{sec2} 
Consider a reflexive sc-Banach space E, that is a Banach space E with an sc-structure which consists of reflexive banach spaces $E_m$.
In this first section we create an sc-structure for $sC^\infty(\mathbb{R},E)$, the space of sc-smooth functions from $\mathbb{R}$ to the sc-Banach space E.  To achieve that we need a nested sequence of Banach spaces $G_m$ such that $G_\infty = sC^\infty(\mathbb{R},E)$ where $G_n \hookrightarrow G_m$ is compact for $n > m$ and $G_\infty$ is dense in all $G_m$.

If E is a Hilbert space, one can use the definition made in \cite{shift_map} at the end of its third section, where weak differentiability in a Pettis integral sense is used.

Here we will use Bochner integrability instead, to get a notion of weak differentiability and Sobolev spaces that form the nested sequence $G_m$. It will be similar and more general than the notion of a sobolev space in \cite{Kre}.

\subsection{Notes on Bochner integrable functions} 

In this subsection we recall several useful properties of the spaces of Bochner-integrable functions $L^p(I,X)$. The main source used is \cite{Kre}, but there are several others covering this topic, like \cite{Vector_measures, PDE_App}.
Let $I \subset \mathbb{R}$ be connected and $(I,\Sigma,\mu)$ be a measure space, where $\mu$ is a $\sigma$-finite measure, for example the Lebesgue measure. This convention for $(I,\Sigma,\mu)$ will be used throughout section \ref{sec2}, unless otherwise stated. The notation in the following sections will be to write $dt$ instead of $d\mu(t)$.

 Let X be a Banach space. The space $L^p(I,X)$ of Bochner-integrable functions with the norm $\Vert u \Vert_{L^p(I,X)} = \Vert \Vert u(.) \Vert_X \Vert_{L^p(I,\mathbb{R})} = (\int_I \Vert u(t) \Vert_X^p dt)^{1/p}$ is a Banach space \cite{Kre, Vector_measures}.

\begin{lemma}\cite[Cor. 2.23]{Kre} \label{reflexive} 

If X is reflexive and $1 < p < \infty$ then $L^p(I,X)$ is reflexive.
\end{lemma}

\begin{lemma} \cite[Prop 2.10]{Kre} \label{sHolder} (skalar H\"{o}lder)

If $u \in L^p(I,X)$, $\varphi \in L^{p'}(I,\mathbb{R})$ where $\frac{1}{p} + \frac{1}{p'} = 1$ we have:
\begin{equation}
\Vert u \cdot \varphi \Vert_{L^1(I,X)} \leq \Vert u \Vert_{L^p(I,X)} \cdot \Vert \varphi \Vert_{L^{p'}(I,\mathbb{R})}
\end{equation}
\end{lemma}

\begin{proof}
$u \cdot \varphi$ is measureable since it is the limit of simple functions, as the pointwise product of simple functions converging to $u$ and $\varphi$ respectively.
With the H\"{o}lder inequality on $L^1(I,\mathbb{R})$ we get
\[
\begin{aligned}
\Vert u \cdot \varphi \Vert_{L^1(I,X)} = \Vert \Vert u \cdot \varphi \Vert_X \Vert_{L^1(I,\mathbb{R})} = \Vert \Vert u \Vert_X \cdot \vert \varphi \vert \Vert_{L^1(I,\mathbb{R})} \leq \\
\leq \Vert \Vert u \Vert_X  \Vert_{L^p(I,\mathbb{R})} \cdot \Vert \varphi \Vert_{L^{p'}(I,\mathbb{R})}
\end{aligned}
\]
\end{proof}

Together with $ C_c^\infty(I,\mathbb{R}) \subset L^{p'}(I,\mathbb{R})$ we will use this lemma for the next one:

\begin{lemma} \label{convergence}
If $u_m$ converges to $u$ in $L^p(I,X)$, it follows that $\int_I u_m(t)\varphi(t) dt$ converges to $\int_I u(t)\varphi(t) dt$ in X for any $\varphi \in C_c^\infty(I,\mathbb{R})$.
\end{lemma}

\begin{proof}
\[
\begin{aligned}
& \Big\Vert \int_I u_m(t)\varphi(t) dt - \int_I u(t)\varphi(t) dt \; \Big\Vert^{}_X = \Big\Vert \int_I (u_m(t) - u(t))\varphi(t) dt \; \Big\Vert^{}_X \leq \\
&\leq\int_I \Vert (u_m(t) - u(t))\varphi(t) \Vert^{}_X \, dt = \Vert (u_m - u)\varphi \Vert^{}_{L^1(I,X)} \leq \\
&\leq \Vert u_m - u \Vert^{}_{L^p(I,X)} \cdot \Vert \varphi \Vert^{}_{L^{p'}(I,\mathbb{R})}
\end{aligned}
\]
\end{proof}

\begin{rmk} \cite{Kre}
If X is a Hilbert space, then $L^2(I,X)$ is also a Hilbert space with the inner product $(u,v)_{L^2(I,X)} := \int_I (u(t),v(t))^{}_X dt$
\end{rmk}

\begin{lemma}\cite[Prop. 2.15]{Kre} \label{dense} ~\par
Let $I \subseteq \mathbb{R}$ be open and take the Lebesgue measure on $I$. If $1 \leq p < \infty$ and $u \in L^p(I,X)$, then there exists a sequence $(\phi_j) \subset C_c^\infty(I,X)$ converging to $u$ in $L^p(I,X)$.
\end{lemma}

Since $L^p(I,X) = L^p(\mathring{I},X) = L^p(\bar{I},X)$ for both the interior and closure of $I$, we will be able to use this lemma later on for $I = [a,b]$. For general $I \subset \mathbb{R}$ we wil use the notation $C_c^\infty(I,X) := C_c^\infty(\mathring{I},X)$.

\begin{rmk} \label{weighted_L}
Now let us define weighted spaces of Bochner integrable functions. This definition has already been made in \cite{weighted_Bochner} for $I = [a,b]$. Let $I$ have the Lebesgue measure. Let $\omega$ be a locally bounded weight function on I, i.e. $\omega: I \rightarrow [0,\infty)$ is a locally bounded measurable function.
The space $L_\omega^p(I,X)$ is the space of all strongly measurable functions u on I such that 
\begin{equation}
\Vert u \Vert_{L_\omega^p(I,X)}^p := \int_I \Vert u(t) \Vert_X^p \omega(t)^p dt < \infty
\end{equation}
This condition on u is equivalent to the condition that $\omega u \in L^p(I,X)$.
We can also view these weighted Bochner spaces as Bochner spaces endowed with the measure $\nu$ where $\nu(A) := \int_A \omega(t)^p dt = \int_A \omega(t)^p d\mu(t)$ for any $A \in \Sigma$.
Thus, as $\omega$ is locally bounded, and therefore $\nu$ is a $\sigma$-finite measure\footnote{Take a countable family $(A_j)$ of compact sets (of finite measure $\mu(A_j)$), which cover I.  $\nu(A_j) = \int_{A_j} \omega(t)^p d\mu(t) < \infty$ since $\omega$ is locally bounded and therefore bounded on compact sets. }, $L_\omega^p(I,X)$ has all the properties that $L^p(I,X)$ has.
\end{rmk}

\subsection{Bochner-Sobolev spaces}  

\begin{defn} ~\par
Let X, Y be Banach spaces where X is embedded in Y, i.e. $X \hookrightarrow Y$.
A $v \in L^1(I,Y)$ is called a \textbf{weak k-th derivative} of $u \in L^1(I,X)$ if 
\begin{equation}
\int_I \varphi^{(k)}(t) u(t) dt = (-1)^k \int_I \varphi(t) v(t) dt
\end{equation}
for all $\varphi \in C_c^\infty(I,\mathbb{R})$, where $\varphi^{(k)}$ is the k-th derivative of $\varphi$.

We denote $v(t) = u^{(k)}(t)$. 
\end{defn}

Let us denote $u = u^{(0)}$. 
As the product of the functions $\varphi \in C_c^\infty(I,\mathbb{R})$ and $u \in C^1(I,X)$ is Fr\'{e}chet differentiable, we get that such a $u$ is weakly differentiable.

\begin{lemma} \label{uniqueness}
The weak k-th derivative is unique.
\end{lemma}

\begin{proof}
Let $u \in L^1(I,X)$, let $v_1, v_2 \in L^1(I,Y)$ be two weak k-th derivatives of $u$, we therefore get:
\[
\int_I \varphi(t) v_1(t) dt = \int_I \varphi^{(k)}(t) u(t) dt = \int_I \varphi(t) v_2(t) dt
\]
so we have $0 = \int_I \varphi(t) (v_1(t)-v_2(t)) dt$ and since this holds for every $\varphi \in C_c^\infty(I,\mathbb{R})$, we get that $v_1(t) - v_2(t) = 0$ almost everywhere. Therefore we have that $v_1 = v_2 \in L^1(I,Y)$.
\end{proof}

\begin{cor}
$u^{(n)}$ be the n-th weak derivative of $u \in L^1(I,X)$, then for all $k + l = n$ we have $u^{(n)} = (u^{(k)})^{(l)}$.
\end{cor}

\begin{defn} ~\par
Let $X_0, . . . ,X_n$ be Banach spaces with $X_n \hookrightarrow X_{n-1} \hookrightarrow . . . \hookrightarrow X_0$. Denote $\hat{X} := (X_n, . . . X_0)$ and $\hat{p} := (p_n, . . . p_0) \in [1,\infty]^{n+1}$. 
The \textbf{n-th Bochner-Sobolev space} is 
\begin{equation}
\begin{aligned}
W^{n,\hat{p}}(I,\hat{X}) := \{ u \in L^{p_n}(I,X_n) \mid \, & u^{(k)} \text{exists for } 1 \leq k \leq n \\
							& \text{and it lies in }  L^{p_{n-k}}(I,X_{n-k}) \}
\end{aligned}
\end{equation}
It is a normed vector space with the norm 
\begin{equation}
\Vert u \Vert_{W^{n,\hat{p}}(I,\hat{X})} := \sum_{k = 0}^n \Vert u^{(k)} \Vert^{}_{L^{p_{n-k}}(I,X_{n-k})}
\end{equation}
\end{defn}

Although the notation $u^{(k)} \in L^{p_{n-k}}(I,X_{n-k})$ is a bit tedious in the upcoming proofs, it will avoid confusion when we will transition to sc-smooth Banach spaces.

\begin{rmk} ~\par
\begin{enumerate}
\item Sometimes we will write $W^{n,\hat{p}}(I;X_n, . . . ,X_0)$ instead of $W^{n,\hat{p}}(I,\hat{X})$.
\item For the cases of $n = 1$ and $n = 2$, this definition has already been made, for example in \cite[chapter 7]{PDE_App}.
\item For $p_0 = . . . = p_n =: p$ and $X_0 = . . . = X_n =: X$ we write the resulting Bochner-Sobolev space as $W^{n,p}(I,X)$. This simpler version of a Bochner-Sobolev space has already been defined in \cite{Kre}. 
\item For $p_0 = . . . = p_n =: p$ we get an embedding $W^{n,p}(I,X_n) \hookrightarrow W^{n,p}(I,\hat{X})$.
\item We get an embedding $W^{n,\hat{p}}(I,\hat{X}) \hookrightarrow W^{n,p}(I,X_0)$, where $\displaystyle{p = \min_k p_k}$.
\end{enumerate}
\end{rmk}

\begin{prop}
$W^{n,\hat{p}}(I,\hat{X})$ is a Banach space.
\end{prop}

\begin{proof}
Let $u_l$ be a Cauchy sequence in $W^{n,\hat{p}}(I,\hat{X})$, i.e. for all $\epsilon > 0$ and all l, m large enough we have $\Vert u_l - u_m \Vert_{W^{n,\hat{p}}(I,\hat{X})} < \epsilon$.
We therefore have that $\Vert u_l^{(k)} - u_m^{(k)} \Vert^{}_{L^{p_{n-k}}(I,X_{n-k})} < \epsilon$ for all $0 \leq k \leq n$. \\
Since $\displaystyle{L^{p_{n-k}}(I,X_{n-k})}$ is complete, we get $\displaystyle{\lim_{l \rightarrow \infty} u_l^{(k)} = u^k}$. Let us set $u^0 =: u$. With Lemma \ref{convergence} we get 
\[
\lim_{l \rightarrow \infty} \int_I u_l^{(k)}(t) \varphi(t) dt =  \int_I u^k(t) \varphi(t) dt
\]
and 
\[
\lim_{l \rightarrow \infty} \int_I u_l(t) \varphi^{(k)}(t) dt =  \int_I u \varphi^{(k)}(t) dt
\]
for any $\varphi \in C_c^\infty(I,\mathbb{R})$. Since the left sides of these two equations are equal and the weak derivatives are unique (Lemma \ref{uniqueness}), we get that u is k-times weakly differentiable and $u^{(k)} = u^k$. Hence $\displaystyle{\lim_{l \rightarrow \infty} u_l = u}$ in $W^{n,\hat{p}}(I,\hat{X})$.
\end{proof}

\begin{rmk}
If the $\hat{X}$ are Hilbert spaces, then $W^{n,2}(I,\hat{X})$, i.e. $p_k = 2$ for all $k \in \{0, . . . n\}$, is also a Hilbert space with the inner product 
\[(u,v)_{W^{n,2}(I,\hat{X})} := \sum_{k=0}^n (u^{(k)},v^{(k)})_{L^2(I,X_{n-k})}^{}\]
This inner product induces the norm 
\[\Vert u \Vert_{W^{n,2}(I,\hat{X})} := \left( \sum_{k=0}^n \Vert u^{(k)}\Vert^2_{L^2(I,X_{n-k})}\right)^{1/2}\]
which is equivalent to the norm previously defined on $W^{n,2}(I,\hat{X})$.
\end{rmk}


\begin{prop} \label{reflexive2}
Let $\hat{X}$ be reflexive Banach spaces and $\hat{p} \in (1,\infty)^{n+1}$. Then $W^{n,\hat{p}}(I,\hat{X})$ is reflexive.
\end{prop}

\begin{proof}
By Lemma \ref{reflexive} we know that $L^{p_k}(I,X_k)$ are reflexive Banach spaces. Define an isometry
\[ \begin{aligned}
T: W^{n,\hat{p}}(I,\hat{X}) &\longrightarrow L^{p_n}(I,X_n) \times . . . \times L^{p_0}(I,X_0) \\
u &\longmapsto (u, . . ., u^{(n)})
\end{aligned} \]
$T(W^{n,\hat{p}}(I,\hat{X})) \subset L^{p_n}(I,X_n) \times . . . \times L^{p_0}(I,X_0)$ is a closed subspace of a reflexive Banach space and therefore also reflexive. Since T is an isometry, and hence an isometric isomorphism onto its image, $W^{n,\hat{p}}(I,\hat{X})$ is reflexive.
\end{proof}

\begin{lemma} \label{continuous} \cite[Lemma 7.1]{PDE_App}

Let the embedding of $X_1$ in $X_0$ be continuous, $\hat{X} = (X_1,X_0)$, $\hat{p} \in [1,\infty]^2$. Let $I =  [a,b]$ be endowed with the Lebesgue measure. Then the embedding $W^{1,\hat{p}}(I,\hat{X}) \hookrightarrow C(I,X_0)$ is continuous.
\end{lemma}
Note that in \cite{PDE_App} $I := [0,T]$ instead of a more general $I = [a,b]$, but the proof can still be done in the same way with very minor adjustments.

\begin{prop} \label{WinC}
Let the embeddings of $X_{k+1}$ in $X_k$ be continuous for all $0 \leq k \leq n$, $\hat{X} = (X_{n+1}, . . . ,X_0)$, $\hat{p} \in [1,\infty]^{n+1}$ and $I = [a,b]$ with the Lebesgue measure. Then we get that 
\[ W^{n+1,\hat{p}}(I,\hat{X}) \subset C^l(I,X_{n-l}) \]
for all $l \in \{0,...,n\}$.
\end{prop}

\begin{proof}
Using the definition of $W^{n+1,\hat{p}}(I,\hat{X})$ we get that 
\[ W^{n+1,\hat{p}}(I;X_{n+1}, . . . ,X_0) \subset W^{l+1,\hat{p}}(I;X_{n+1}, . . . ,X_{n-l})  \subset W^{l+1,p}(I,X_{n-l}) \]
 where $\displaystyle{p = \min_{0\leq k \leq l} p_k}$. Also we know that $u \in W^{l+1,p}(I,X_{n-l})$ if and only if $u^{(k)} \in W^{1,p}(I,X_{n-l})$ for all $0 \leq k \leq l$.
Using proposition 3.8 in \cite{Kre} we therefore get that $u^{(k)}$ is differentiable almost everywhere for all $0 \leq k \leq l$. Using lemma \ref{continuous}, we get that $u^{(k+1)}$ is continuous for all $0 \leq k \leq l-1$ and therefore $u^{(k)}$ is differentiable for all $0 \leq k \leq l-1$, i.e. $u \in C^l(I,X_{n-l})$.
\end{proof}

\begin{rmk} \label{WinC2}
If $I \subset \mathbb{R}$ is not compact, we can use that for every $t_0 \in I$ we have an $\epsilon > 0$, so that $[t_0 - \epsilon, t_0 + \epsilon] \subset I$. Hence we have for every $u$, that $u\vert^{}_{[t_0 - \epsilon, t_0 + \epsilon]} \in W^{n+1,\hat{p}}([t_0 - \epsilon, t_0 + \epsilon],\hat{X}) \subset C^l([t_0 - \epsilon, t_0 + \epsilon],X_{n-l})$, if $u \in W^{n+1,\hat{p}}(I,\hat{X})$. Since Fr\'{e}chet differentiability is a local condition we know that $u$ is therefore k-times differentiable on the whole $I$.
\end{rmk}

\subsection{Compact embeddings}  

\begin{lemma} [Ehrlings lemma] \cite[p.58]{Lions} \cite[Lemma 7.6]{PDE_App}

Let X, Y, Z be Banach spaces where X is compactly embedded in Y and Y embedds continuously into Z, i.e. $\displaystyle{X \xhookrightarrow[]{cpt} Y \hookrightarrow Z}$. Then for any $\epsilon > 0$, there exists a constant $C_\epsilon$ such that 
\begin{equation} \label{ehrling}
\Vert x \Vert^{}_Y \leq \epsilon \Vert x \Vert^{}_X + C_\epsilon \Vert x \Vert^{}_Z
\end{equation}
\end{lemma}

The next theorem will provide the compact embedding of $W^n$ in $W^{n-1}$. This theorem only works for $I =[a,b]$ with the Lebesgue measure. Thus we will later need to consider weighted Bochner-Sobolev spaces to extend this compactness result for $I = \mathbb{R}$. The following theorem will be an extended version of the Aubin-Lions lemma, for the original version see for example \cite[chap.1, thm 5.1]{Lions}.

\begin{thm} [Aubin-Lions lemma] ~\par
Let $\hat{X}=(X_n, . . . , X_0), \hat{Y}=(Y_{n-1}, . . . , Y_0)$ be Banach spaces, where $\hat{X}$ are reflexive. Let the embedding of $X_k$ in $Y_{k-1}$ be compact for all $k \in \{1, . . . , n\}$ and the embedding of $Y_k$ in $X_k$ be continuous for all $k \in \{0, . . . , n-1\}$, that is
\[ X_n \xhookrightarrow[]{cpt} Y_{n-1} \hookrightarrow X_{n-1} \xhookrightarrow[]{cpt} \: .\: .\: .\: .\: . \hookrightarrow X_1 \xhookrightarrow[]{cpt} Y_0 \hookrightarrow X_0 \]
where all embeddings are continuous.
Let $I =[a,b]$ with the Lebesgue measure.
Let $\hat{p} = (p_n, . . . , p_0)$ and $\tilde{p} = (p_n, . . . , p_1)$.

Then the embedding 
\[ W^{n,\hat{p}}(I,\hat{X}) \hookrightarrow W^{n-1,\tilde{p}}(I,\hat{Y}) \]
is compact.
\end{thm}

\begin{proof}
Suppose $(u_l) \subset W^{n,\hat{p}}(I,\hat{X})$ is a bounded sequence. Since $W^{n,\hat{p}}(I,\hat{X})$ is a reflexive Banach space by Proposition \ref{reflexive2}, there exists a weak convergent subsequence $u_l \rightharpoonup u$. (In this proof all subsequences of $u_l$ will still be named $u_l$.) W.l.o.g. $u = 0$, otherwise we look at $u_l - u$ instead.

Using Ehrlings lemma for $\displaystyle{X_{n-k} \xhookrightarrow[]{cpt} Y_{n-k-1} \hookrightarrow X_{n-k-1}}$, we get for any $\epsilon > 0$ a $C_{\epsilon,k} > 0$ such that 
\[ \begin{aligned}
\Vert u_l^{(k)} \Vert^{}_{L^{p_{n-k}}(I,Y_{n-k-1})} &= \big\Vert \Vert u_l^{(k)} \Vert^{}_{Y_{n-k-1}} \big\Vert^{}_{L^{p_{n-k}}(I,\mathbb{R})} \\
\leq& \big\Vert \epsilon \Vert u_l^{(k)} \Vert^{}_{X_{n-k}} + C_{\epsilon,k} \Vert u_l^{(k)} \Vert^{}_{X_{n-k-1}} \big\Vert^{}_{L^{p_{n-k}}(I,\mathbb{R})} \\
\leq& \ \epsilon \big\Vert \Vert u_l^{(k)} \Vert^{}_{X_{n-k}} \big\Vert^{}_{L^{p_{n-k}}(I,\mathbb{R})} +C_{\epsilon,k} \big\Vert \Vert u_l^{(k)} \Vert^{}_{X_{n-k-1}} \big\Vert^{}_{L^{p_{n-k}}(I,\mathbb{R})} \\
=& \ \epsilon \Vert u_l^{(k)} \Vert^{}_{L^{p_{n-k}}(I,X_{n-k})} +C_{\epsilon,k} \Vert u_l^{(k)} \Vert^{}_{L^{p_{n-k}}(I,X_{n-k-1})}
\end{aligned} \]
Let us now take $\displaystyle{C_\epsilon := \max_k (C_{\epsilon,k})}$ and let $\Vert u_l \Vert_{W^{n,\hat{p}}(I,\hat{X})} \leq \tilde{C}$. We thus get 
\begin{align*}
\Vert u_l^{(k)} \Vert^{}_{W^{n-1, \tilde{p}}(I,\hat{Y})} &= \sum_{k = 0}^{n-1} \Vert u_l^{(k)} \Vert^{}_{L^{p_{n-k}}(I,Y_{n-k-1})}  \\
&\leq \sum_{k = 0}^{n-1} \epsilon \Vert u_l^{(k)} \Vert^{}_{L^{p_{n-k}}(I,X_{n-k})} +C_{\epsilon,k} \Vert u_l^{(k)} \Vert^{}_{L^{p_{n-k}}(I,X_{n-k-1})} \\
&\leq \epsilon \ \tilde{C} + C_\epsilon \sum_{k = 0}^{n-1} \Vert u_l^{(k)} \Vert^{}_{L^{p_{n-k}}(I,X_{n-k-1})}
\end{align*}

It is therefore left to prove that $\displaystyle{\lim_{l \rightarrow \infty} \Vert u_l^{(k)} \Vert^{}_{L^{p_{n-k}}(I,X_{n-k-1})} = 0} \, $ for all \\ $k \in \{0, . . . , n-1\}$.

As $u_l \rightharpoonup 0$ in $W^{n,\hat{p}}(I,\hat{X})$, using the isometry from the proof of proposition \ref{reflexive2}, we also have $u_l^{(k)} \rightharpoonup 0$ in $L^{p_{n-k}}(I,X_{n-k})$. 

Fix $t_0 \in I = [a,b]$ and define $v_l^k: [0,1] \rightarrow X_k$, $v_l^k(t) := u_l^{(k)}(t_0 + \lambda t)$, with $\lambda$ so small that $t_0 + \lambda < b$. Hence 
\begin{equation} v_l^k (0) = u_l^{(k)} (t_0) \end{equation}
\begin{equation} \begin{aligned}
\Vert v_l^k \Vert^{}_{L^{p_{n-k}}(0,1;X_{n-k})} &= \left( \int_0^1 \Vert v_l^k(t) \Vert^{p_{n-k}}_{X_{n-k}} dt \right)^{1/{p_{n-k}}} \\
&= \left( \int_{t_0}^{t_0+\lambda} \Vert u_l^{(k)}(s) \Vert^{p_{n-k}}_{X_{n-k}} \lambda^{-1} ds \right)^{1/{p_{n-k}}} \\
&\leq \lambda^{-1/p_{n-k}} \left( \int_a^b \Vert u_l^{(k)}(s) \Vert^{p_{n-k}}_{X_{n-k}} ds \right)^{1/{p_{n-k}}} \\
&= \lambda^{-1/p_{n-k}} \Vert u_l^{(k)} \Vert^{}_{L^{p_{n-k}}(I,X_{n-k})} \\
&\leq \tilde{C} \lambda^{-1/p_{n-k}}
\end{aligned} \end{equation}
and for the first weak derivative of $v_l^k$:
\begin{equation} \label{ineq2} \begin{aligned}
\Vert (v_l^k)' \Vert^{}_{L^{p_{n-k-1}}(0,1;X_{n-k-1})} &= \left( \int_0^1 \Vert (v_l^k(t))' \Vert^{p_{n-k-1}}_{X_{n-k-1}} dt \right)^{1/{p_{n-k-1}}} \\
&= \left( \int_{t_0}^{t_0+\lambda} \Vert \lambda \cdot u_l^{(k+1)}(s) \Vert^{p_{n-k-1}}_{X_{n-k-1}} \lambda^{-1} ds \right)^{1/{p_{n-k-1}}} \\
&\leq \lambda^{1-\frac{1}{p_{n-k-1}}} \left( \int_a^b \Vert u_l^{(k+1)}(s) \Vert^{p_{n-k-1}}_{X_{n-k-1}} ds \right)^{1/{p_{n-k-1}}} \\
&= \lambda^{1-\frac{1}{p_{n-k-1}}} \Vert u_l^{(k+1)} \Vert^{}_{L^{p_{n-k-1}}(I,X_{n-k-1})} \\
&\leq \tilde{C} \lambda^{1-\frac{1}{p_{n-k-1}}}
\end{aligned} \end{equation}

Now consider a $\phi \in C^1([0,1],\mathbb{R})$ with $\phi(0) = -1$ and $\phi(1) = 0$, then
\[ v_l^k(0) = \int_0^1 \frac{d}{dt}(\phi(t) v_l^k(t)) dt =\int_0^1 \phi'(t) v_l^k(t) dt + \int_0^1 \phi(t) (v_l^k)'(t) dt \]
Using the skalar H\"{o}lder inequality (lemma \ref{sHolder}), $\Vert \phi \Vert_{L^q(0,1;\mathbb{R})} =: c$ where \\ $\frac{1}{p_{n-k-1}} + \frac{1}{q} = 1$ and the inequality (\ref{ineq2}), we therefore get
\begin{equation} \label{ineq3} \begin{aligned}
\Vert v_l^k(0) \Vert_{X_{n-k-1}}^{} \leq& \int_0^1 \Vert \phi(t) (v_l^k)'(t) \Vert_{X_{n-k-1}} dt + \left\Vert \int_0^1 \phi'(t) v_l^k(t) dt \right\Vert_{X_{n-k-1}} \\
\leq& \Vert \phi \Vert_{L^q(0,1;\mathbb{R})} \cdot \Vert (v_l^k)' \Vert^{}_{L^{p_{n-k-1}}(0,1;X_{n-k-1})} \\
&+ \left\Vert \int_0^1 \phi'(t) v_l^k(t) dt \right\Vert_{X_{n-k-1}} \\
\leq& \: C \cdot \lambda^{1-\frac{1}{p_{n-k-1}}} + \left\Vert \int_0^1 \phi'(t) v_l^k(t) dt \right\Vert_{X_{n-k-1}} \\
\end{aligned} \end{equation}
Now we need to get the second term on the RHS of (\ref{ineq3}) to be small. We test the integral with elements $x'$ of the dual space of $X_{n-k}$.
\[ \begin{aligned}
\langle x' , \int_0^1 \phi'(t) v_l^k(t) dt \rangle^{}_{X_{n-k}} &= \int_0^1 \langle \phi'(t) x' , v_l^k(t) \rangle^{}_{X_{n-k}} dt \\
&= \lambda^{-1}  \int_a^b \langle \chi_{[t_0,t_0+\lambda]}(s) \phi'\left(\frac{s-t_0}{\lambda}\right) x' , u_l^{(k)}(s) \rangle^{}_{X_{n-k}} ds \\
&\xrightarrow[l \rightarrow \infty]{} 0
\end{aligned} \]
The convergence to 0 happens since $u_l^{(k)} \rightharpoonup 0$ in $L^{p_{n-k}}(I,X_{n-k})$. 
We have thus shown that 
\[ \int_0^1 \phi'(t) v_l^k(t) dt \rightharpoonup 0  \;\;\;\;\;\; \text{   in } X_{n-k}. \]
Because of $X_{n-k} \xhookrightarrow[]{cpt} Y_{n-k-1} \hookrightarrow X_{n-k-1}$ we have $X_{n-k} \xhookrightarrow[]{cpt} X_{n-k-1}$ and thus we get
\[ \int_0^1 \phi'(t) v_l^k(t) dt \rightarrow 0  \;\;\;\;\;\; \text{   in } X_{n-k-1}. \]

For any $\epsilon > 0$ we can now choose $\lambda > 0$ in such a way, that \\ $\frac{\epsilon}{2} \geq \: C \cdot \lambda^{1-\frac{1}{p_{n-k-1}}}$.
Then we can choose $l$ large enough to get \\ $\frac{\epsilon}{2} \geq \left\Vert \int_0^1 \phi'(t) v_l^k(t) dt \right\Vert_{X_{n-k-1}}$.
With the definition of $v_l^k$ and the inequality (\ref{ineq3}) we thus get

\[ \begin{aligned} 
\Vert u_l^{(k)}(t_0) \Vert_{X_{n-k-1}}^{} &= \Vert v_l^k(0) \Vert_{X_{n-k-1}}^{} \\
 &\leq \: C \cdot \lambda^{1-\frac{1}{p_{n-k-1}}} + \left\Vert \int_0^1 \phi'(t) v_l^k(t) dt \right\Vert_{X_{n-k-1}} \leq \epsilon 
\end{aligned} \]
We have thus shown that $u_l^{(k)}$ converges pointwise almost everywhere to 0 in $X_{n-k-1}$, i.e. $\displaystyle{ \lim_{l \rightarrow \infty} \Vert u_l^{(k)}(t) \Vert_{X_{n-k-1}}^{} = 0}$ a.e..

With lemma \ref{continuous} we know that the embedding \\
 $W^{1,(p_{n-k},p_{n-k-1})}(I,(X_{n-k},X_{n-k-1})) \hookrightarrow C(I,X_{n-k-1})$ is continuous.\\
$(u_l^{(k)}) \subset W^{1,(p_{n-k},p_{n-k-1})}(I,(X_{n-k},X_{n-k-1}))$ is bounded, since \\ $(u_l) \subset W^{n,\hat{p}}(I,\hat{X})$ is bounded. 
Hence $\Vert u_l^{(k)}(t) \Vert_{X_{n-k-1}}^{} \leq \: C$, where C is independent of t and l.
With the dominated convergence theorem we therefore get
\[ \begin{aligned}
\lim_{l \rightarrow \infty} \Vert u_l^{(k)} \Vert^{}_{L^{p_{n-k}}(I,X_{n-k-1})} &= \lim_{l \rightarrow \infty} \left( \int_I \Vert u_l^{(k)}(t) \Vert^{p_{n-k}}_{X_{n-k-1}} dt \right)^{1/p_{n-k}} \\
&= \left( \int_I \left(  \lim_{l \rightarrow \infty} \Vert u_l^{(k)}(t) \Vert^{}_{X_{n-k-1}} \right)^{p_{n-k}} dt \right)^{1/p_{n-k}} = 0
\end{aligned} \]
\end{proof}

Since $W^{n,\hat{p}}(I,\hat{X}) = W^{n,\hat{p}}(int(I),\hat{X}) =W^{n,\hat{p}}(\bar{I},\hat{X})$ the Aubin-Lions lemma could be used for any bounded and connected $I \subset \mathbb{R}$. The condition that $I$ be compact was only used, so that we can use lemma \ref{continuous}.

\begin{defn} ~\par
Let $\omega: I \rightarrow [0,\infty)$ be a locally bounded weight function, $I$ being endowed with the Lebesgue measure. Let $\hat{X}$ and $\hat{p}$ be as in the definition of Bochner-Sobolev spaces. The \textbf{n-th weighted Bochner-Sobolev space} is 
\begin{equation}
\begin{aligned}
W^{n,\hat{p}}_\omega(I,\hat{X}) := \{ u \in L^{p_n}_\omega(I,X_n) \mid & u^{(k)} \text{exists for } 1 \leq k \leq n \\
							& \text{and it lies in }  L^{p_{n-k}}_\omega(I,X_{n-k}) \}
\end{aligned}
\end{equation}
It is a normed vector space with the norm 
\begin{equation}
\Vert u \Vert_{W^{n,\hat{p}}_\omega(I,\hat{X})} := \sum_{k = 0}^n \Vert u^{(k)} \Vert^{}_{L^{p_{n-k}}_\omega(I,X_{n-k})}
\end{equation}
\end{defn}

\begin{rmk} 
 As seen in remark \ref{weighted_L}, we can view these weighted Bochner-Sobolev spaces as Bochner-Sobolev spaces with another measure $\nu$ on I. Since this measure $\nu$ is $\sigma$-finite, weighted Bochner-Sobolev spaces have all the properties which Bochner-Sobolev spaces have.
 
Note that by using lemma \ref{continuous}, we get that every $u \in W^{n,\hat{p}}_\omega(I,\hat{X})$ is continuous, but the embedding $W^{n,\hat{p}}_\omega(I,\hat{X}) \hookrightarrow C^0(I,X_0)$ might not be continuous, depending on the weight $\omega$.
 \end{rmk}

\begin{defn} ~\par
A sequence of continuous weight functions $\omega_n: \mathbb{R} \rightarrow (0,\infty), \;\; n \in \mathbb{N}_0$ is called \textbf{scaled weights} if for all $m,n \in \mathbb{N}_0$, such that $m < n$, there exist $T \geq 0$ and $D > 0$ such that
\begin{enumerate}[i)]
\item $\displaystyle{\lim_{t \rightarrow \pm\infty} \frac{\omega_m(t)}{\omega_n(t)} = 0}$
\item for all $t_1 > t_2 > T$ we have $\frac{\omega_m(t_1)}{\omega_n(t_1)} < D \: \frac{\omega_m(t_2)}{\omega_n(t_2)}$
\item for all $t_1 < t_2 < -T$ we have $\frac{\omega_m(t_1)}{\omega_n(t_1)} < D \: \frac{\omega_m(t_2)}{\omega_n(t_2)}$
\end{enumerate}
\end{defn}
This defintion could equivalently be done with $m = n-1$.
Note that all conditions except the continuity only dictate behavior outside of a compact set, as continuous weights that differ only on a compact set define the same weighted (Bochner-)Sobolev space.

\begin{exa} \label{exweights}
Exponential weights are smooth scaled weights. \\
Take a smooth monotone cutoff function $\beta \in C^\infty(\mathbb{R},[-1,1])$ such that $\beta(s) = -1$ for $s\leq -1$ and $\beta(s) = 1$ for $s \geq 1$.
Take constant $\delta \in \mathbb{R}$ and define the exponential weight function $\omega_\delta: \mathbb{R} \rightarrow \mathbb{R}$ as
\[ \omega_\delta(s) := e^{\delta \beta(s) s} \]
Note that usually $\delta > 0$ is chosen, like in \cite[Example 3.9]{shift_map}, so that the associated weighted Sobolev space only contains exponentially decaying functions.\\
Take a sequence $(\delta_n) \subset \mathbb{R}$ so that $\delta_m < \delta_n$ for $m<n$. Then $\omega_n :=\omega_{\delta_n}$ are scaled weights: \\
\[
\frac{\omega_m(s)}{\omega_n(s)} = e^{(\delta_m - \delta_n) \beta(s) s} \xrightarrow{s \rightarrow \pm \infty} 0
\]
and this is monotone for at least $\vert s \vert \geq 1$, therefore $T = 1$ and $D = 1$.
Note that even a sequence of exponential weights with changing $\beta$ would be scaled weights.
\end{exa}

\begin{exa}
We also have other sequences of weights which are scaled weights:
\begin{enumerate}[1)]
\item Even exponent polynomial weights $\omega_n(s) := s^{2n} + c$, $c > 0$

\item Polynomial weights $\omega_n(s) := (\beta(s)s)^{l_n} + c$, with $c>0$ and $(l_n) \subset [1,\infty)$ monotone increasing

\item We can also use weights, where a $D > 1$ is neccesary to fulfill conditions ii) and iii) of the definition, that is weights, that are not monotonely increasing/decreasing as $s \rightarrow \infty$.
We can modify existing scaled weights to achieve this, e.g. $\omega_n(s) =(1+c+\cos(s+n))(\beta(s)s)^{l_n} + c$ or a variation of the exponential weight $\omega_n(s) :=(2+\cos(s+n)) e^{\delta_n \beta(s) s}$.
\end{enumerate}
\end{exa}

\begin{prop} \label{wcompact} ~\par
Let $\hat{X}=(X_n, . . . , X_0), \hat{Y}=(Y_{n-1}, . . . , Y_0)$ be Banach spaces, where $\hat{X}$ are reflexive. Let the embedding of $X_k$ in $Y_{k-1}$ be compact for all $k \in \{1, . . . , n\}$ and the embedding of $Y_k$ in $X_k$ be continuous for all $k \in \{0, . . . , n-1\}$, that is
\[ X_n \xhookrightarrow[]{cpt} Y_{n-1} \hookrightarrow X_{n-1} \xhookrightarrow[]{cpt} \: .\: .\: .\: .\: . \hookrightarrow X_1 \xhookrightarrow[]{cpt} Y_0 \hookrightarrow X_0 \]
where all embeddings are continuous.
Let $0 \in I \subset \mathbb{R}$ with the Lebesgue measure.
Let $\hat{p} = (p_n, . . . , p_0)$ and $\tilde{p} = (p_n, . . . , p_1)$.
Let $(\omega_n)$ be scaled weights.
Then the embedding 
\[ W^{n,\hat{p}}_{\omega_n}(I,\hat{X}) \hookrightarrow W^{n-1,\tilde{p}}_{\omega_{n-1}}(I,\hat{Y}) \]
is compact.
\end{prop}

\begin{proof}
Let $T > 0$ and $D > 0$ be the constants associated to the pair $\omega_n, \omega_{n-1}$ of weights.
Take a bounded sequence in $W^{n,\hat{p}}_{\omega_n}(I,\hat{X})$, i.e. $\Vert u_l \Vert_{W^{n,\hat{p}}_{\omega_n}(I,\hat{X})} \leq \tilde{C}$.
Choose $I_N := [-N,N] \cap I$, $N \in \mathbb{N}$ as an exhaustion by compact sets of $I$. Be $\displaystyle{c_N := \max_{I_N} \omega_n^{-1}(t)}$. Then we have

\[ \begin{aligned}
\Vert u_l \Vert_{W^{n,\hat{p}}(I_N,\hat{X})} &\leq \sum_{k = 0}^n \left( \int_{I_N} \Vert u_l^{(k)}(t) \Vert_{X_{n-k}}^{p_{n-k}} c_N^{p_{n-k}} \: \omega_n^{p_{n-k}}(t) \: dt \right)^{1/p_{n-k}} \\
&\leq c_N \sum_{k = 0}^n \left( \int_I \Vert u_l^{(k)}(t) \Vert_{X_{n-k}}^{p_{n-k}} \omega_n^{p_{n-k}}(t) \: dt \right)^{1/p_{n-k}} \\
&= c_N \Vert u_l \Vert_{W^{n,\hat{p}}_{\omega_n}(I,\hat{X})}\leq c_N \: \tilde{C}
\end{aligned} \]
Using the Aubin-Lions lemma on $u_l\vert_{I_1}^{}$ we get a convergent subsequnce $u_{l_1}\vert_{I_1}^{} \rightarrow v_1$ in $W^{n-1,\tilde{p}}(I_1,\hat{Y})$.
Again using the Aubin-Lions lemma on $u_{l_1}\vert_{I_2}^{}$ we get a convergent subsequence $u_{l_2}\vert_{I_2}^{} \rightarrow v_2$ in $W^{n-1,\tilde{p}}(I_2,\hat{Y})$.
With increasing N we thus get further subsequences  $u_{l_N}\vert_{I_N}^{} \rightarrow v_N$ in $W^{n-1,\tilde{p}}(I_N,\hat{Y})$ and therefore pointwise almost everywhere. 
Note that for $N_1 > N_2$ we know that $v_{N_1}\vert_{I_{N_2}}^{} = v_{N_2}$ in $W^{n-1,\tilde{p}}(I_{N_2},\hat{Y})$. 

Define $u(t) := v_N(t)$ for all $N \in \mathbb{N}$ such that $t \in [-N,N]$, the prospective limit of $u_l$ in $W^{n-1,\tilde{p}}_{\omega_n-1}(I,\hat{Y})$. 
Now take a diagonal subsequence $u_j$ of the subsequences $u_{l_i}$. This diagonal subsequence $u_j$ converges to $u$ pointwise almost everywhere. Using Fatou's lemma on $\Vert u_j^{k} \Vert_{Y_{n-k-1}}^{p_{n-k}}$ for $0 \leq k \leq n-1$ we get that $u \in W^{n-1,\tilde{p}}_{\omega_{n}}(I,\hat{Y})$.

Let $\epsilon > 0$, choose $K > T > 0$ with $K \in \mathbb{N}$ large enough that the following two inequalities hold 
\begin{align*}
&\Vert u \Vert_{W^{n-1,\tilde{p}}_{\omega_{n-1}}(I\backslash I_K,\hat{Y})} \\
&= \sum_{k=0}^{n-1} \left( \int_{I\backslash I_K} \left(\frac{\omega_{n-1}(t)}{\omega_n(t)}\right)^{p_{n-k}} \: \omega_n^{p_{n-k}}(t) \: \Vert u^{(k)}(t) \Vert_{Y_{n-k-1}}^{p_{n-k}} dt \right)^{1/p_{n-k}} \\
&\leq \sum_{k=0}^{n-1}  \frac{\omega_{n-1}(K)}{\omega_n(K)} \left( \int_{I\backslash [-K,K]} \omega_n^{p_{n-k}}(t) \: \Vert u^{(k)}(t) \Vert_{Y_{n-k-1}}^{p_{n-k}} dt \right)^{1/p_{n-k}} \\
&\leq \frac{\omega_{n-1}(K)}{\omega_n(K)} \sum_{k=0}^{n-1} \left( \int_I \omega_n^{p_{n-k}}(t) \: \Vert u^{(k)}(t) \Vert_{Y_{n-k-1}}^{p_{n-k}} dt \right)^{1/p_{n-k}} \\
&\leq \frac{\omega_{n-1}(K)}{\omega_n(K)} \sum_{k=0}^{n-1} \liminf_{j \rightarrow \infty} \left( \int_I \omega_n^{p_{n-k}}(t) \: \Vert u_j^{(k)}(t) \Vert_{Y_{n-k-1}}^{p_{n-k}} dt \right)^{1/p_{n-k}} \\
&\leq \frac{\omega_{n-1}(K)}{\omega_n(K)} \sum_{k=0}^{n-1} \liminf_{j \rightarrow \infty} \left( \int_I \omega_n^{p_{n-k}}(t) \: C_{n-k}^{p_{n-k}} \Vert u_j^{(k)}(t) \Vert_{X_{n-k}}^{p_{n-k}} dt \right)^{1/p_{n-k}} \\
&\leq \frac{\omega_{n-1}(K)}{\omega_n(K)} \sum_{k=0}^{n-1} C_{n-k} \: \tilde{C} < \frac{\epsilon}{4}
\end{align*}
In the above inequality we use Fatou's lemma, the definition of scaled weights and the continuous embedding $X_{n-k} \hookrightarrow Y_{n-k-1}$. The next inequality is similar, except for the use of Fatou's lemma.
\begin{align*}
&\Vert u_j \Vert_{W^{n-1,\tilde{p}}_{\omega_{n-1}}(I\backslash I_K,\hat{Y})} \\
&= \sum_{k=0}^{n-1} \left( \int_{I\backslash I_K} \left(\frac{\omega_{n-1}(t)}{\omega_n(t)}\right)^{p_{n-k}} \: \omega_n^{p_{n-k}}(t) \: \Vert u_j^{(k)}(t) \Vert_{Y_{n-k-1}}^{p_{n-k}} dt \right)^{1/p_{n-k}} \\
&\leq \frac{\omega_{n-1}(K)}{\omega_n(K)} \sum_{k=0}^{n-1} \left( \int_I \omega_n^{p_{n-k}}(t) \: \Vert u_j^{(k)}(t) \Vert_{Y_{n-k-1}}^{p_{n-k}} dt \right)^{1/p_{n-k}} \\
&\leq \frac{\omega_{n-1}(K)}{\omega_n(K)} \sum_{k=0}^{n-1} \left( \int_I \omega_n^{p_{n-k}}(t) \: C_{n-k}^{p_{n-k}} \Vert u_j^{(k)}(t) \Vert_{X_{n-k}}^{p_{n-k}} dt \right)^{1/p_{n-k}} \\
&\leq \frac{\omega_{n-1}(K)}{\omega_n(K)} \sum_{k=0}^{n-1} C_{n-k} \: \tilde{C} < \frac{\epsilon}{4}
\end{align*}
Since $u_j \rightarrow u$ in $W^{n-1,\tilde{p}}(I_K,\hat{Y})$ and $\omega_{n-1}\vert_{I_K}^{}$ is bounded, we know there exists a $J$ large enough that for all $j \geq J$ we have
\[ \Vert u_j - u \Vert_{W^{n-1,\tilde{p}}_{\omega_{n-1}}(I_K,\hat{Y})} \leq C_K \Vert u_j - u \Vert_{W^{n-1,\tilde{p}}(I_K,\hat{Y})} < \frac{\epsilon}{2} \]
Altogether we thus have 
\[ \begin{aligned}
&\Vert u_j - u \Vert_{W^{n-1,\tilde{p}}_{\omega_{n-1}}(I,\hat{Y})} \\ 
&\leq \Vert u_j - u \Vert_{W^{n-1,\tilde{p}}_{\omega_{n-1}}(I_K,\hat{Y})} + \Vert u \Vert_{W^{n-1,\tilde{p}}_{\omega_{n-1}}(I\backslash I_K,\hat{Y})} + \Vert u_j \Vert_{W^{n-1,\tilde{p}}_{\omega_{n-1}}(I\backslash I_K,\hat{Y})} < \epsilon
\end{aligned} \]
\end{proof}

\subsection{Smooth functions are dense in $W^{n,p}(I,\hat{X})$}  

In this subsection we will extend lemma \ref{dense} to Bochner-Sobolev spaces, that is, we will show a version of the Meyers-Serrin theorem and also that compactly supported functions mapping to $X_n$ are dense in $W^{n,p}(I,\hat{X})$.

From now on we will require that $\hat{p} = (p, . . . ,p)$ and that the embeddings of $X_k \hookrightarrow X_{k-1}$ are continuous for all $X_k$ in $\hat{X}$. We will write $p$ instead of $\hat{p}$ from now on.

The first half of this subsection is heavily based on section 4.2 in \cite{Kre}, but there are some small adjustments that have to be made.

\begin{lemma} \label{4.8} \cite[compare Lemma 4.8]{Kre} ~\par
Let $\phi \in C_c^\infty(I,\mathbb{R})$ and $u \in W^{n,\hat{p}}(I,\hat{X})$. Then $\phi u \in W^{n,\hat{p}}(I,\hat{X})$ and the weak derivatives of $\phi u$ are given by the usual Leibniz formula
\[ (\phi u)^{(k)} = \sum_{j=0}^k {k \choose j} \phi^{(k-j)} u^{(j)} \]
\end{lemma}

\begin{proof}
First view the case $k=1$. Let $\psi \in C_c^\infty(I,\mathbb{R})$. Then $\phi\psi \in C_c^\infty(I,\mathbb{R})$. We thus have
\[ -\int_I (u \phi)' \psi dt + \int_I u \phi' \psi dt = \int_I u \phi \psi' + u \phi' \psi dt =  \int_I u (\phi\psi)' dt = -\int_I u' \phi\psi dt \]
With the linearity of the integral and the uniqueness of the weak derivative we therefore have $(u\phi)' = u' \phi + u \phi'$.

Inductively we now procede for $k>1$:
\[ \begin{aligned}
\int_I u \phi \psi^{(k)} dt &= -\int_I (u \phi)' \psi^{(k-1)} dt = (-1)^k\int_I (u'\phi + u\phi')^{(k-1)} \psi \: dt \\
&= (-1)^k \int_I \left(\sum_{j=0}^{k-1} {k-1 \choose j} u^{(j+1)} \phi^{(k-1-j)} + u^{(j)} \phi^{(k-j)}\right) \psi \: dt \\
&= (-1)^k \int_I \sum_{j=0}^{k} {k \choose j} u^{(j)} \phi^{(k-j)}  \psi \: dt \\
\end{aligned} \]
Since all derivatives of $\phi \in C_c^\infty(I,\mathbb{R})$ are also functions in $C_c^\infty(I,\mathbb{R})$ and $u^{(j)} \in L^p(I,X_{n-j}) \hookrightarrow L^p(I,X_{n-k})$, we know that $ (u \phi)^{(k)} \in W^{n,\hat{p}}(I,\hat{X})$, as the embedding $L^p(I,X_{n-j}) \hookrightarrow L^p(I,X_{n-k})$ is continuous, since $X_{n-j} \hookrightarrow X_{n-k}$ is continous.
\end{proof}

Let $\eta \in C_c^\infty(\mathbb{R},[0,\infty))$ be a positive mollifier, i.e. $supp(\eta) \subset [-1,1]$ and $\int_\mathbb{R} \eta dt = \int_{[-1,1]} \eta dt = 1$. Let $\eta_\epsilon(t) := \frac{1}{\epsilon} \eta(\frac{t}{\epsilon})$ for any $\epsilon > 0$. It still holds that $\int_\mathbb{R} \eta_\epsilon dt = 1$ and $\eta_\epsilon \in C_c^\infty(\mathbb{R},[0,\infty))$, but $supp(\eta_\epsilon) \subset [-\epsilon,\epsilon]$.
Define the convolution of $u \in L_{loc}^1(I,X)$ with $\eta_\epsilon$ as
\[ \begin{aligned}
u \ast \eta_\epsilon (t) := \int_\mathbb{R} u(t-s) \eta_\epsilon(s) ds = \int_\mathbb{R} u(s) \eta_\epsilon(t-s) ds =: \eta_\epsilon \ast u (t)
\end{aligned} \] 
Since the convolution, in the form of the second integral, only depends on $\eta_\epsilon$ in $t$, we get that $u \ast \eta_\epsilon \in C^\infty(\mathbb{R},X)$.
We use the distributional version of the support of a function u:
\[
supp \: u := I \Big\backslash \bigcup_{U \subset I open, \: u\vert_U^{}=0} U 
\]

\begin{lemma} \label{supp} \cite[Lemma 4.4]{Kre} \[
supp \: u \ast \eta_\epsilon \subset supp \: u + supp \: \eta_\epsilon = supp \: u + [-\epsilon,\epsilon]
\] \end{lemma}

\begin{lemma} \label{4.6} \cite[Theorem 4.6]{Kre} ~\par
Let $1 \leq p < \infty$ and $u \in L^p(\mathbb{R},X)$. As $\epsilon \rightarrow 0$ we have that 
\[ \Vert u - u \ast \eta_\epsilon \Vert^{}_{L^p(\mathbb{R},X)} \rightarrow 0 \]
\end{lemma}

\begin{lemma} \label{4.7} \cite[compare Proposition 4.7]{Kre} ~\par
Let $u \in W^{n,p}(I,\hat{X})$ , let $k \leq n$, then
\[   (u \ast \eta_\epsilon)^{(k)}(t) =  u^{(k)} \ast \eta_\epsilon (t) \]
for all $t \in I$ such that $dist(t,\partial I) > \epsilon$.
\end{lemma}

\begin{proof}
Since $dist(t,\partial I) > \epsilon$, we get that $\eta_\epsilon(t-.) \in C_c^\infty(I,\mathbb{R})$. We thus have
\[ \begin{aligned}
(u \ast \eta_\epsilon)^{(k)}(t) &= \partial^k_t \int_I  \eta_\epsilon(t-s) \: u(s) ds = \int_I \partial^k_t \eta_\epsilon(t-s) \: u(s) ds \\
 &= (-1)^k \int_I \partial^k_s \eta_\epsilon(t-s) \: u(s) ds = \int_I  \eta_\epsilon(t-s) \: \partial^k_s u(s) ds \\
 &=  u^{(k)} \ast \eta_\epsilon (t)
\end{aligned} \]
\end{proof}

\begin{lemma} \label{convolution} ~\par
Let $u \in L^p(I,X)$ , let $k \in \mathbb{N}$, then
\[   (\eta_\epsilon \ast u)^{(k)}(t) =   \eta_\epsilon^{(k)} \ast u (t) \]
for all $t \in I$ such that $dist(t,\partial I) > \epsilon$.
\end{lemma}

\begin{proof}
Since $dist(t,\partial I) > \epsilon$, we get that $\eta_\epsilon(t-.) \in C_c^\infty(I,\mathbb{R})$. We thus have
\[ \begin{aligned}
(\eta_\epsilon \ast u)^{(k)}(t) &= \partial^k_t \int_I  \eta_\epsilon(t-s) \: u(s) ds = \int_I \partial^k_t \eta_\epsilon(t-s) \: u(s) ds = \eta_\epsilon^{(k)} \ast u(t) 
\end{aligned} \]
\end{proof}

\begin{thm} [Meyers-Serrin]  \cite[compare Theorem 4.11]{Kre} ~\par
$C^\infty(I,X_n) \cap W^{n,p}(I,\hat{X})$ is dense in $W^{n,p}(I,\hat{X})$ for $1\leq p < \infty$.
\end{thm}

\begin{proof}
Be $\delta > 0$ and $u \in W^{n,p}(I,\hat{X})$.

Let $I_l := \{t \in I \, \vert \: l > \vert t \vert , \: dist(t,\partial I) > \frac{1}{l} \}$ be for all $l \in \mathbb{N}$ and set $I_0 = I_{-1} = \emptyset$. Define $U_l := I_{l+1}\backslash \overline{I_{l-1}}$ open subsets of $I$ for all $l \in \mathbb{N}_0$.
Let $\rho_l \in C_c^\infty(I,\mathbb{R})$ be a smooth partition of unity, subordinate to the open cover $U_l$ of $I$.
Let $(\epsilon_l)_{l \in \mathbb{N}}$ be a sequence such that $\epsilon_l \leq \frac{1}{(l+2)^2}$.
Using lemma \ref{supp} we therefore have $supp ( \eta_{\epsilon_l} \ast (\rho^{}_l \:u)) \subset I_{l+2}\backslash I_{l-2}$ since $dist(I_{l+1},\partial I_{l+2}) < \frac{1}{(l+2)^2}$ and $dist(I_{l-2},\partial I_{l-1}) < \frac{1}{(l+2)^2}$. Using the lemmata \ref{4.8}, \ref{4.6} and \ref{4.7} with $\epsilon_l$ small enough for lemma \ref{4.6} we get
\begin{equation} \label{ineqMS} \begin{aligned}
\Vert \eta_{\epsilon_l} \ast (\rho^{}_l \:u) - \rho^{}_l \:u \Vert_{W^{n,p}(I,\hat{X})} &= \sum_{k=0}^n \Vert \partial^k(\eta_{\epsilon_l} \ast (\rho^{}_l \:u)) - \partial^k(\rho^{}_l \:u) \Vert_{L^p(I,X_{n-k})}\\
&= \sum_{k=0}^n \Vert \eta_{\epsilon_l} \ast \partial^k (\rho^{}_l \:u) - \partial^k (\rho^{}_l \:u) \Vert_{L^p(I,X_{n-k})} \\
&< \sum_{k=0}^n \frac{\delta}{(n+1) 2^{l+1}} = \frac{\delta}{2^{l+1}}
\end{aligned} \end{equation}
Choose $\displaystyle{v(t) = \sum_{l=0}^\infty \eta_{\epsilon_l} \ast (\rho^{}_l \:u)}$. This is actually a locally finite sum, since $supp ( \eta_{\epsilon_l} \ast (\rho^{}_l \:u)) \subset I_{l+2}\backslash I_{l-2} \,$ and $I_{l+2}\backslash I_{l-2} \cap I_{l+2+m}\backslash I_{l-2+m} = \emptyset \,$ for all $\vert m \vert > 4$, $m\in \mathbb{Z}$. Hence $v \in C^\infty(I,X_n)$ is a locally finite sum of functions in $C^\infty(I,X_n)$.
\[ \begin{aligned}
\Vert v - u \Vert_{W^{n,p}(I,\hat{X})} &= \Vert  \sum_{l=0}^\infty \eta_{\epsilon_l} \ast (\rho^{}_l \:u) - \rho^{}_l \:u \Vert_{W^{n,p}(I,\hat{X})} \\
&\leq  \sum_{l=0}^\infty \Vert \eta_{\epsilon_l} \ast (\rho^{}_l \:u) - \rho^{}_l \:u \Vert_{W^{n,p}(I,\hat{X})} \\
&< \sum_{l=0}^\infty \frac{\delta}{2^{l+1}} = \delta
\end{aligned} \]
We therefore get that $v \in W^{n,p}(I,\hat{X})$ and since we can construct such a $v$ for any $\delta > 0$, we get a sequence of smooth functions converging to $u$ in $W^{n,p}(I,\hat{X})$. 
\end{proof}

\begin{cor} \label{weightedMS} ~\par
$C^\infty(I,X_n) \cap W_\omega^{n,p}(I,\hat{X})$ is dense in $W_\omega^{n,p}(I,\hat{X})$ for a continuous weight $\omega$ and ${1\leq p < \infty}$.
\end{cor}

\begin{proof}
The proof of this corollary works entirely analogous to the proof of Meyers-Serrin, only the inequality (\ref{ineqMS}) needs to adjusted.

Using the fact that $\eta_{\epsilon_l} \ast (\rho^{}_l \:u) - \rho^{}_l \:u$ has support in $I_{l+2}\backslash I_{l-2}$ and therefore only in $U_{l-1} \cup U_l \cup U_{l+1} =: W$, the analog of inequality (\ref{ineqMS}) therefore looks like
\begin{align*}
&\Vert \eta_{\epsilon_l} \ast (\rho^{}_l \:u) - \rho^{}_l \:u \Vert_{W_\omega^{n,p}(I,\hat{X})} \\
&= \sum_{k=0}^n \Vert \eta_{\epsilon_l} \ast \partial^k (\rho^{}_l \:u) - \partial^k (\rho^{}_l \:u) \Vert_{L^p_\omega(I,X_{n-k})} \\
&= \sum_{k=0}^n \Vert \omega (\eta_{\epsilon_l} \ast \partial^k (\rho^{}_l \:u) - \partial^k (\rho^{}_l \:u)) \Vert_{L^p(I,X_{n-k})} \\
&= \sum_{k=0}^n \Vert \sum_{j=l-1}^{l+1} \rho_j \omega (\eta_{\epsilon_l} \ast \partial^k (\rho^{}_l \:u) - \partial^k (\rho^{}_l \:u)) \Vert_{L^p(I,X_{n-k})} \\
&\leq \sum_{k=0}^n \sum_{j=l-1}^{l+1} \sup_W( \rho_j \omega ) \Vert \eta_{\epsilon_l} \ast \partial^k (\rho^{}_l \:u) - \partial^k (\rho^{}_l \:u) \Vert_{L^p(I,X_{n-k})} \\
&\leq \sum_{k=0}^n \sup_W( \omega ) \Vert \eta_{\epsilon_l} \ast \partial^k (\rho^{}_l \:u) - \partial^k (\rho^{}_l \:u) \Vert_{L^p(I,X_{n-k})} \\
&< \sum_{k=0}^n \frac{\delta}{(n+1) 2^{l+1}} = \frac{\delta}{2^{l+1}}
\end{align*}
\end{proof}

\begin{lemma} \label{Ccdense} ~\par
$C_c^\infty(I,X_n)$ is dense in $W^{n,p}_{\omega_n}(I,\hat{X})$ for a continuous weight function $\omega_n$ and ${1\leq p < \infty}$.
\end{lemma}

\begin{proof}
Let $u \in W^{n,p}_{\omega_n}(I,\hat{X})$, w.l.o.g. $u \in C^\infty(I,X_n) \cap W^{n,p}_{\omega_n}(I,\hat{X})$ using corollary \ref{weightedMS}.

Take a compact exhaustion $(I_K)$ of I and smoothed indicator functions $\psi_K$ subordinate to $I_K$, i.e. $\psi_K \in C_c^\infty(I,[0,1])$ such that $supp \:\psi_K \subset I_{K+1}$ and $\psi_K\vert^{}_{I_K} =1$.  Therefore also $\psi_K^{(k)}\vert^{}_{I_K} =0$ and $supp (\psi_K^{(k)}) \subset \bar{I}_{K+1} \backslash \mathring{I}_K$.

Be $\delta > 0$. Choose a $\tilde{\delta} > 0$ so that it satisfies the inequality 
\[ \sum_{k=0}^n \tilde{\delta} + \sum_{k=0}^n  \sum_{j=0}^{k-1} {k \choose j} \max_{I_{K+1}\backslash \mathring{I}_K}(\psi_K^{(k-j)}) \: \tilde{\delta} \leq \delta\]
for all $K > K_0$.
Using lemma \ref{dense}, there exist $\phi_k \in C_c^\infty(I,X_{n-k})$, such that $\Vert u^{(k)} - \phi_k \Vert_{L^p(I,X_{n-k})} < \tilde{\delta}$. Then for a $K > K_0$ large enough we get that $supp\: \phi_k \subset I_K$ for all $0 \leq k \leq n$. Using lemma \ref{4.8} we thus have
\begin{align*}
&\:\Vert u - \psi_K u\Vert_{W^{n,p}_{\omega_n}(I,E_n)} = \sum_{k=0}^n \Vert u^{(k)} - (\psi_K u)^{(k)}\Vert_{L^p_{\omega_n}(I,E_{n-k})} \\
&= \sum_{k=0}^n \left\Vert u^{(k)} - \sum_{j=0}^k {k \choose j} \psi_K^{(k-j)} u^{(j)}\right\Vert_{L^p_{\omega_n}(I,E_{n-k})} \\ 
&\leq \sum_{k=0}^n \left\Vert u^{(k)} - \sum_{j=0}^{k-1} {k \choose j} \psi_K^{(k-j)} u^{(j)} - \psi_K u^{(k)} \right\Vert_{L^p_{\omega_n}(I\backslash I_K,E_{n-k})} \\ 
&\:\:\:\:\:\:\:\:\:\:\:\:\:\:\:\:\:\:\:\:+ \sum_{k=0}^n \Vert u^{(k)} - 0 - 1\cdot u^{(k)}\Vert_{L^p_{\omega_n}(I_K,E_{n-k})} \\ 
&\leq \sum_{k=0}^n \Vert u^{(k)} - \psi_K u^{(k)} \Vert_{L^p_{\omega_n}(I\backslash I_K,E_{n-k})} \\
&\:\:\:\:\:\:\:\:\:\:\:\:\:\:\:\:\:\:\:\: + \sum_{k=0}^n  \sum_{j=0}^{k-1} \left\Vert {k \choose j} \psi_K^{(k-j)} u^{(j)} \right\Vert_{L^p_{\omega_n}(I\backslash I_K,E_{n-k})} \\
&\leq \sum_{k=0}^n \Vert u^{(k)} -0 \Vert_{L^p_{\omega_n}(I\backslash I_K,E_{n-k})} \\
&\:\:\:\:\:\:\:\:\:\:\:\:\:\:\:\:\:\:\:\: + \sum_{k=0}^n  \sum_{j=0}^{k-1} {k \choose j} \max_{I_{K+1}\backslash \mathring{I}_K}(\psi_K^{(k-j)}) \Vert u^{(j)} \Vert_{L^p_{\omega_n}(I\backslash I_K,E_{n-k})} \\
&= \sum_{k=0}^n \Vert u^{(k)} - \phi_k \Vert_{L^p_{\omega_n}(I\backslash I_K,E_{n-k})} \\
&\:\:\:\:\:\:\:\:\:\:\:\:\:\:\:\:\:\:\:\: + \sum_{k=0}^n  \sum_{j=0}^{k-1} {k \choose j} \max_{I_{K+1}\backslash \mathring{I}_K}(\psi_K^{(k-j)}) \Vert u^{(j)} - \phi_k \Vert_{L^p_{\omega_n}(I\backslash I_K,E_{n-k})} \\
&<  \sum_{k=0}^n \tilde{\delta} + \sum_{k=0}^n  \sum_{j=0}^{k-1} {k \choose j} \max_{I_{K+1}\backslash \mathring{I}_K}(\psi_K^{(k-j)}) \: \tilde{\delta} \leq \delta
\end{align*}

As $u \in C^\infty(I,X_n) \cap W^{n,p}_{\omega_n}(I,\hat{X})$ we have $\psi_K u \in C_c^\infty(I,X_n)$ converging to $u$. Hence $C_c^\infty(I,X_n)$ is dense in $C^\infty(I,X_n) \cap W^{n,p}_{\omega_n}(I,\hat{X})$ and therefore, using corollary \ref{weightedMS}, dense in $W^{n,p}_{\omega_n}(I,\hat{X})$.
\end{proof}

\subsection{sc-structure with Bochner-Sobolev spaces}  
We use the notion of sc-smoothness as outlined in \cite{scdef} by Hofer, Wysocki and Zehnder. 
Let E be an sc-Banach space, i.e. the Banach space E possesses an sc-structure of Banach spaces $E_n$. $I \subset \mathbb{R}$ has its trivial, that is constant, sc-structure.

\begin{lemma} \label{sc=c} ~\par
Let $I \subset \mathbb{R}$ open and let $E$ be an sc-Banach space. Then $f: I \rightarrow E$ is sc-smooth, if and only if $f: I \rightarrow E_n$ is smooth for all $E_n$ in the sc-structure of $E$.
\end{lemma}

\begin{proof}
Be f sc-smooth. Using proposition 2.14 in \cite{scdef}, we get that for all $k,n \in \mathbb{N}$ we have that $f: I_{n+k} = I \rightarrow E_n$ is of class $C^k$.

Be $f: I \rightarrow E_n$ smooth for all $k \in \mathbb{N}$. Then f is $sc^0$, since $f: I_n = I \rightarrow E_n$ are all continuous. Using proposition 2.15 in \cite{scdef} on $f: I_{n+k} = I \rightarrow E_n$, we get that f is sc-smooth.
\end{proof}

Let us denote $sC^\infty(F,E) = \{f:F \rightarrow E \:\text{sc-smooth}\}$ and $\hat{E}_n := (E_n, . . . E_0)$.

\begin{lemma} \label{ccdense} ~\par
Let E be an sc-Banach space. Then $C_c^\infty(I,E_\infty)$ is dense in $W^{n,p}_{\omega_n}(I,\hat{E}_n)$ for all $n \in \mathbb{N}$.
\end{lemma}

\begin{proof}
Let $u \in W^{n,p}_{\omega_n}(I,\hat{E}_n)$, w.l.o.g. $u \in C_c^\infty(I,E_n)$ using lemma \ref{Ccdense}. Be $\delta > 0$. Denote $J := supp\:u \subset I$ compact.

Choose $\tilde{\delta} > 0$ so that $\int_J \tilde{\delta} \omega_n(t) \: dt < \frac{\delta}{4}$ .
As $E_\infty \subset E_n$ is dense, we can choose a sequence $(u_l)$ of functions converging pointwise to $u$, with $u_l(t) \in E_\infty$ and $supp\:u_l = J$. Take $N \in \mathbb{N}$ large enough, that for all $l>N$ we have $\Vert u(t) - u_l(t) \Vert_{E_n} < \tilde{\delta}$ for all $t \in J$.
\begin{align*}
\Vert u - u_l \Vert_{L^p_{\omega_n}(I,E_n)} = \int_I \Vert u(t) - u_l(t) \Vert^{}_{E_n} \omega_n(t) \: dt \leq \int_J \tilde{\delta} \omega_n(t) \: dt < \frac{\delta}{4}
\end{align*}

We therefore also have $u_l \in L^p_{\omega_n}(I,E_n)$. Hence we can mollify $u_l$ to
\[u_{l,\epsilon} := \eta_\epsilon \ast u_l \]
Since $u_{l,\epsilon} \in C^\infty(I,E_\infty)$ and, the fact that $supp \, u_{l,\epsilon} \subset J + [-\epsilon,\epsilon] =: J_\epsilon$ is compact, using the lemma \ref{supp}, we get $u_{l,\epsilon} \in C_c^\infty(I,E_\infty)$.

Using lemmata \ref{4.6}, \ref{4.7} and \ref{convolution} we therefore have
\begin{align*}
&\Vert u - u_{l,\epsilon} \Vert_{W^{n,p}_{\omega_n}(I,\hat{E}_n)} \leq \Vert u - \eta_\epsilon \ast u \Vert_{W^{n,p}_{\omega_n}(I,\hat{E}_n)} + \Vert \eta_\epsilon \ast u - u_{l,\epsilon} \Vert_{W^{n,p}_{\omega_n}(I,\hat{E}_n)} \\
&= \sum_{k=0}^n \Vert \partial^k( u - \eta_\epsilon \ast u) \Vert_{L^p_{\omega_n}(I,E_{n-k})} + \sum_{k=0}^n \Vert \partial^k( \eta_\epsilon \ast (u - u_l)) \Vert_{L^p_{\omega_n}(I,E_{n-k})} \\
&= \sum_{k=0}^n \Vert \partial^k u - \eta_\epsilon \ast \partial^k u) \Vert_{L^p_{\omega_n}(I,E_{n-k})} +\sum_{k=0}^n \Vert \eta_\epsilon^{(k)} \ast (u - u_l)) \Vert_{L^p_{\omega_n}(J_\epsilon,E_{n-k})} \\
&< \frac{\delta}{2} + \sum_{k=0}^n \int_{J_\epsilon} \omega_n(t) \left\Vert \int_\mathbb{R} \eta_\epsilon^{(k)} (t-s) (u(s) - u_l(s)) ds \right\Vert_{E_{n-k}} dt \\
&\leq \frac{\delta}{2} + \sum_{k=0}^n \int_{J_\epsilon} \omega_n(t) \int_\mathbb{R} \vert \eta_\epsilon^{(k)} (t-s) \vert \: \Vert u(s) - u_l(s)\Vert^{}_{E_{n-k}} ds \: dt \\
&\leq \frac{\delta}{2} + \sum_{k=0}^n \int_{J_\epsilon} \omega_n(t) \int_\mathbb{R} \vert \eta_\epsilon^{(k)} (t-s) \vert \: \tilde{\delta} \: ds \: dt \\
&\leq \frac{\delta}{2} + \sum_{k=0}^n \int_{J_\epsilon} \omega_n(t) C_{\epsilon,k} \tilde{\delta} \: dt < \delta
\end{align*}
Note that in this inequality we first need to estimate the left summand with $\epsilon \rightarrow 0$ and then the right one with $l \rightarrow \infty$.

Altogether we have shown that the inclusion $C_c^\infty(I,E_\infty) \subset C_c^\infty(I,E_n)$ is dense, when viewed as an embedding of dense subspaces of $W^{n,p}_{\omega_n}(I,\hat{E}_n)$.
Hence, using lemma \ref{Ccdense}, we get that $C_c^\infty(I,E_\infty) \subset W^{n,p}_{\omega_n}(I,\hat{E}_n)$ is dense.
\end{proof}

Now we can formulate the final theorem of this first section to get an sc-structure using Bochner-Sobolev spaces.

\begin{thm} \label{scBS} ~\par
Let $E$ be a reflexive sc-Banach space, i.e. all Banach spaces $E_n$ in the sc-structure are reflexive. Let $I \subseteq \mathbb{R}$ and let $\omega_n$ be scaled weights. Denote $\hat{E}_n := (E_n, . . . ,E_0)$.

Then $G := L^p_{\omega_0}(I,E_0)$ has an sc-structure, where $G_n = W^{n,p}_{\omega_n}(I,\hat{E}_n)$.
\end{thm}

\begin{proof}
The embedding $G_n \hookrightarrow G_m$ is compact for $m < n$ because of proposition \ref{wcompact}, where we set $\hat{X} = \hat{E}_n$ and $\hat{Y} = \hat{E}_{n-1}$.

Because of lemma \ref{sc=c} and lemma \ref{ccdense} we have that
\begin{align*}
G_\infty = \bigcap_{n \in \mathbb{N}_0} G_n \supset  \bigcap_{n \in \mathbb{N}_0} (G_n \cap C^\infty(I,E_n)) = \bigcap_{n \in \mathbb{N}_0} G_n \cap sC^\infty(I,E) \\
\supset \bigcap_{n \in \mathbb{N}_0} C_c^\infty(I,E_n) \supset C_c^\infty(I,E_\infty)
\end{align*}
and $C_c^\infty(I,E_\infty)$ is dense in every $G_n$ and thus also in $G_\infty$. Hence $G_\infty$ is dense in every $G_n$.
\end{proof}

In the case that $\bar{I} \subset \mathbb{R}$ is compact, i.e. $I$ is bounded in $\mathbb{R}$, we could leave out the scaled weights, as we can use the Aubin-Lions lemma and the Meyers-Serrin theorem directly and get the same results. But as the scaled weights are continuous on $\mathbb{R}$ and strictly positive we get that 
\[ W^{n,p}_{\omega_n}(I,\hat{E}_n) = W^{n,p}(I,\hat{E}_n)\]
anyways, so we do not need another theorem stating the same fact.

As always with sc-Banach spaces, we can choose another space out of the sc-structure as ''starting'' space, for example set $G = G_0 := W^{1,p}_{\omega_1}(I,\hat{E}_1) =W^{1,p}_{\omega_1}(I;E_1,E_0)$. This is used as we can start in a space of functions of higher regularity, but still retain the same ''smooth space'' $G_\infty$.

\begin{exa}
As defined in example \ref{exweights}, take exponential weights $\omega_n(s) := e^{\delta_n \beta(s) s}$ with $\delta_0 < \delta_n < \delta_\infty < \infty$, i.e. $(\delta_n) \subset \mathbb{R}$ a monotone increasing sequence converging to $\delta_\infty$. Then the exponentially weighted Bochner-Sobolev spaces $W^{n,p}_{\omega_n}(I,\hat{E}_n) =: W^{n,p}_{\delta_n}(I,\hat{E}_n)$ form a sc-structure on the Bochner space $L^p_{\delta_0}(I,E_0)$.

\end{exa}

\begin{prop} \label{scinfty}  ~\par
For $G_\infty$ as in theorem \ref{scBS} we get that
\begin{equation}
G_\infty \subset sC^\infty(I,E) 
\end{equation}
\end{prop}

\begin{proof}
Let $u \in G_\infty = \bigcap_{n \in \mathbb{N}_0} G_n$, that is $u \in G_n = W^{n,p}_{\omega_n}(I,\hat{E}_n)$ for all $n \in \mathbb{N}$.
Using proposition \ref{WinC} or remark \ref{WinC2} we therefore get that $u \in C^k(I,E_l)$ for any $k,l \in \mathbb{N}$. Hence we have that $u \in C^\infty(I,E_l)$ for all $l \in \mathbb{N}$ and using lemma \ref{sc=c} we thus have $u \in sC^\infty(I,E)$.
\end{proof}

Note that in general, at least for unbounded $I\in \mathbb{R}$, we have that  \[ G_\infty \neq sC^\infty(I,E) \]

We will denote $G_\infty = sC_{\omega^{}_\infty}^\infty(I,E)$.

\newpage

\section{The length functional}   

Let (M,g) be an oriented n-dimensional Riemannian manifold. Let $S^1 = \mathbb{R}/\mathbb{Z}$. \\

\subsection{The length functional on loops}  

Take the length functional on the free loop space, defined as
\begin{align*}
 L: C^\infty(S^1,M) &\longrightarrow \mathbb{R} \\
 v &\longmapsto \int_{S^1} \Vert v'(t) \Vert^{}_g dt
\end{align*}
or alternatively using the weak differential instead for all $k \in \mathbb{N}$
\begin{align*}
 L: W^{k,2}(S^1,M) &\longrightarrow \mathbb{R} \\
 v &\longmapsto \int_{S^1} \Vert v'(t) \Vert^{}_g dt
\end{align*}
Thus we can also view it as a functional from the sc-Banach space $W^{1,2}(S^1,M)$ to $\mathbb{R}$. 

The length functional is continuous, as the space $W^{k,2}(S^1,M)$ is locally k-weak diffeomorphic to $W^{k,2}\big(S^1,S^1 \times (-r,r)^{n-1}\big)$, where we can choose a local diffeomorphism $\varphi$ to map a loop $v_0$ onto $S^1 \times \{0\}^{n-1}$. Hence $L(\varphi(v))$ is continuous, as it is part of the norm of $\varphi(v) \in W^{1,1}\big(S^1,S^1 \times (-r,r)^{n-1}\big)$ and the embedding of $W^{k,2}\big(S^1,S^1 \times (-r,r)^{n-1}\big)$ in $W^{1,1}\big(S^1,S^1 \times (-r,r)^{n-1}\big)$ is continuous. Since $\varphi$ is a local homeomorphism, we get that L is continuous in $W^{k,2}(S^1,M)$.

Usually one uses the energy functional 
\begin{align*}
 E: W^{k,2}(S^1,M) &\longrightarrow \mathbb{R} \\
 v &\longmapsto \int_{S^1} \Vert v'(t) \Vert^2_g dt 
\end{align*}
instead of the length functional, as the critical points of the energy are true geodesics (and constant loops), and not just reparametrizations of geodesics.
The critical sets of E are finite dimensional manifolds, which have a non-free $S^1$ action on them, with the execption of the constant loops, which form a citical manifold diffeomorphic to M.

The critical sets of L are even worse, than those of the energy, as they are infinite-dimensional and the reparametrizations acting on them are not even a group (these reparametrizations do not need to be injective, only surjective maps $\phi:S^1 \rightarrow S^1$). We will get around this difficulty by only looking at embedded loops (immersed loops might work as well). There only diffeomorphic reparametrizations exist, whose group action we will quotient out, to be able to work on isolated critical points or finite dimensional critical manifolds. 

Instead of the energy E, one could also use the functional $F = \sqrt{E}$, for which we get the inequality $L(v) \leq F(v)$ for all $v \in W^{k,2}(S^1,M)$. This functional has the same critical points as the energy.


\subsection{Reparametrizations}  

Now we will look at the groups of injective reparametrizations of $S^1$, i.e. diffeomorphisms from $S^1$ into itself, under whose action the length functional is invariant. From now on we will omit the injective when refering to these groups. In the next subsection we will look at these reparametrizations on embedded loops to get a space of unparametrized embedded loops.

Let us first look at smooth reparametrizations
\[ \mathcal{D}^\infty := \{ \psi : S^1 \rightarrow S^1 \vert \: \psi \text{ is a smooth diffeomorphism}\} \]
Note that this space, or rather group can also be seen as the product $\mathcal{D}_\infty = O(2) \times Diffeo(p,o,S^1)$, where $Diffeo(p,o,S^1)$ is the group of orientation preserving diffeomorphisms with a fixed point p.

For $k \in \mathbb{N}$ we can analogously define k-times weakly differentiable repara-metrizations
\[ \mathcal{D}^k := \{ \psi : S^1 \rightarrow S^1 \vert \: \psi \text{ is a k-times weak differentiable diffeomorphism}\} \]

The group $\mathcal{D}^\infty$ is a Lie group, as a Fr\'{e}chet manifold, see \cite{inverse_fct, diffS1}.
The other reparametrization groups $\mathcal{D}^k$ are open subsets of $W^{k,2}(S^1,S^1)$, see \cite{diffS1}. Note that they are not Lie groups, as the group multiplication $\phi \psi = \psi \circ \phi$ is not smooth.

Lastly we define the group of continuous reparametrizations
\[ \mathcal{D}^0 := \{ \psi : S^1 \rightarrow S^1 \vert \: \psi \text{ is a homeomorphism}\} \]

Now look at the group actions $f$ of $\mathcal{D}^k$ on the loop space:
\begin{align*}
f: W^{k,2}(S^1,M) \times \mathcal{D}^k &\longrightarrow W^{k,2}(S^1,M) \\
f: C^\infty(S^1,M) \times \mathcal{D}^\infty &\longrightarrow C^\infty(S^1,M) \\
(v,\psi) \:\: &\longmapsto \psi_\ast v := v \circ \psi
\end{align*}
It would be ideal if these group actions would form an sc-smooth action.

\begin{lemma} \label{sc0}  ~\par
The reparametrization actions $f: W^{k,2}(S^1,M) \times \mathcal{D}^k \longrightarrow W^{k,2}(S^1,M)$ as above, for $k\geq 2$, are continuous and 
\begin{align*}
\bar{f}: \Gamma^{k}(TM,v) \times \mathcal{D}^{k} &\longrightarrow \Gamma^{k}(TM,v) \\
(V,\psi) \:\: &\longmapsto \psi_\ast V = V \circ \psi
\end{align*}
are continuous for any $v \in W^{k,2}(S^1,M)$ and for any $k\geq 2$.
\end{lemma}

\begin{proof}
We will prove the continuity of $\bar{f}$, the proof for $f$ will work analogously.
Throughout the proof we will use that $\psi^{(l)}$ and $V^{(l)}$ are continuous for $0 \leq l < k$.
We will also use the fact that $W^{1,2}(S^1,M)$ embedds continuously into $C^0(S^1,M)$ and hence for any bounded $\psi$ and $V$, we get that $\psi^{(l)}$ and $V^{(l)}$ are bounded in $C^0$, i.e. have bounded supremum for $0 \leq l < k$. We also use that as $\psi$ is a diffeomorphism, therefore strictly monotone and hence $\frac{1}{\psi'(t)}$ is bounded.

First we will show the continuity in $V \in \Gamma^{k+1}(TM,v) $. \\
Be $\epsilon > 0$. Let $\Vert V_1 - V_2 \Vert_{\Gamma^{k+1}(TM,v)} < \delta$
\begin{align*}
&\Vert \bar{f}(V_1,\psi) - \bar{f}(V_2,\psi) \Vert_{\Gamma^{k}(TM,v)} = \Vert V_1 \circ \psi - V_2 \circ \psi \Vert_{\Gamma^{k}(TM,v)} \\
&= \sum_{j=0}^k \left( \int_{S^1} \Vert \partial^j (V_1(\psi(t)) -V_2(\psi(t))) \Vert^2_g \: dt\right)^{\frac{1}{2}} \\
&= \sum_{j=0}^k \left( \int_{S^1} \Vert \partial^{j-1} (V_1'(\psi(t)) -V_2'(\psi(t))) \cdot \psi'(t) \Vert^2_g \: dt\right)^{\frac{1}{2}} \\
&= \sum_{j=0}^k \bigg( \int_{S^1}  \Big\Vert \sum_{i=1}^j C_{ij} (V_1^{(j-i+1)}(\psi(t)) - V_2^{(j-i+1)}(\psi(t))) \\
&\:\:\:\:\:\:\:\:\:\:\:\:\:\:\:\:\:\:\:\:\:\:\:\: \cdot P(\psi'(t),...,\psi^{(i)}(t)) \Big\Vert^2_g \: dt \bigg)^{\frac{1}{2}} \\
&\leq \sum_{j=0}^k \bigg( \sum_{i=1}^j C_{ij} \int_{S^1}  \Vert  (V_1^{(j-i+1)}(\psi(t)) - V_2^{(j-i+1)}(\psi(t))) \Vert^2_g  \\
&\:\:\:\:\:\:\:\:\:\:\:\:\:\:\:\:\:\:\:\:\:\:\:\: \cdot P\big(\vert\psi'(t)\vert^2,...,\vert\psi^{(i)}(t)\vert^2\big)  \: dt \bigg)^{\frac{1}{2}} \\
&\leq \sum_{j=0}^k \left( \sum_{i=1}^{j\neq k} C_{ij} \int_{S^1}  \Vert  (V_1^{(j-i+1)}(\psi(t)) - V_2^{(j-i+1)}(\psi(t))) \Vert^2_g \:  C_{i,j,\psi} \: dt \right)^{\frac{1}{2}} \\
&\:\:\:\:\:\:\:\:\:\:\:\: \:\:\:\:\:\:\:\:\:\:\:\: + C_k \left( \int_{S^1}  \Vert  (V_1'(\psi(t)) - V_2'(\psi(t))) \Vert^2_g \cdot \vert\psi^{(k)}(t)\vert^2 \: dt \right)^{\frac{1}{2}} \\
&\leq \sum_{j=0}^k \sum_{i=1}^{j\neq k} \tilde{C}_{i,j,\psi} \Vert V_1 - V_2 \Vert_{\Gamma^{k+1}(TM,v)} + C_k \left( \int_{S^1}  C \delta \cdot \vert\psi^{(k)}(t)\vert^2 \: dt \right)^{\frac{1}{2}} \\
&< \sum_{j=0}^k \sum_{i=1}^{j\neq k} \tilde{C}_{i,j,\psi} \delta + \tilde{C}_{k,\psi} \delta < \epsilon
\end{align*}

Now we show the continuity in $\psi \in \mathcal{D}^{k+1}$. \\
Be $\epsilon > 0$. Let $\Vert \psi_1 -\psi_2 \Vert^{}_{\mathcal{D}^{k+1}} < \delta$ and $\Vert \psi_1 \Vert^{}_{\mathcal{D}^{k+1}}\, , \,\Vert \psi_2 \Vert^{}_{\mathcal{D}^{k+1}} < c$.
We use a family of mollifiers $\eta_\gamma$ for $\gamma > 0$ small enough and we use the continuity of the diffeomorphisms $\psi$ and $\psi'$ to get for $\delta \rightarrow 0$ that $\psi_2^{-1} \circ \psi_1(t) \rightarrow t$ and $D(\psi_2^{-1} \circ \psi_1(t)) \rightarrow 1$. We hence get:
\begin{align*}
&V^{(k)} \circ \psi_1 \circ \psi_1^{-1} \circ \psi_2 \ast \eta_\gamma (t) \\
&= \int_{S^1} V^{(k)} \circ \psi_1 \circ \psi_1^{-1} \circ \psi_2 (s) \cdot \eta_\gamma(t - \psi_2^{-1} \circ \psi_1 \circ \psi_1^{-1} \circ \psi_2 (s)) ds \\
&= \int_{S^1} V^{(k)} \circ \psi_1(s) \cdot \eta_\gamma(t - \psi_2^{-1} \circ \psi_1(s)) \frac{1}{det D(\psi_2^{-1} \circ \psi_1(s))} ds \\
&\xrightarrow{\delta \rightarrow 0} \int_{S^1} V^{(k)} \circ \psi_1(s) \cdot \eta_\gamma(t - s) 1 ds = V^{(k)} \circ \psi_1 \ast \eta_\gamma (t) \\
\end{align*}
Therefore we can show that the following term is small.
\begin{align*}
&\Vert V^{(k)} \circ \psi_1 - V^{(k)} \circ \psi_2 \Vert_{L^2(S^1,M)} \leq \Vert V^{(k)} \circ \psi_1 - V^{(k)} \circ \psi_1 \ast \eta_\gamma \Vert_{L^2(S^1,M)} \\
&+ \Vert V^{(k)} \circ \psi_1 \ast \eta_\gamma - V^{(k)}  \circ \psi_2 \ast \eta_\gamma \Vert_{L^2(S^1,M)} + \Vert V^{(k)} \circ \psi_2 \ast \eta_\gamma - V^{(k)} \circ \psi_2 \Vert_{L^2(S^1,M)} \\
&< 2 \delta_\gamma + \bigg( \int_{S^1} \Big\Vert V^{(k)} \circ \psi_1 \ast \eta_\gamma (t) - V^{(k)} \circ \psi_1 \circ \psi_1^{-1} \circ \psi_2 \ast \eta_\gamma (t) \Big\Vert^2 dt \bigg)^{\frac{1}{2}} \\
&< 2 \delta_\gamma + \delta_{\gamma,\delta}
\end{align*}

With that term and the continuity of all derivatives except the k-th we get:
\begin{align*}
&\Vert \bar{f}(V,\psi_1) - \bar{f}(V,\psi_2) \Vert_{\Gamma^{k}(TM,v)} = \Vert V \circ \psi_1 - V \circ \psi_2 \Vert_{\Gamma^{k}(TM,v)} \\
&= \sum_{j=0}^k \left( \int_{S^1} \Vert \partial^j (V(\psi_1(t)) - V(\psi_2(t))) \Vert^2_g \: dt\right)^{\frac{1}{2}} \\
&= \sum_{j=0}^k \bigg( \int_{S^1}  \Big\Vert \sum_{i=1}^j C_{ij} V^{(j-i+1)}(\psi_1(t)) \cdot P(\psi_1'(t),...,\psi_1^{(i)}(t))  \\
&\:\:\:\:\:\:\:\:\:\:\:\:\:\:\:\:\:\:\:\:\: - V^{(j-i+1)}(\psi_2(t)) \cdot P(\psi_2'(t),...,\psi_2^{(i)}(t)) \Big\Vert^2_g \: dt \bigg)^{\frac{1}{2}} \\
&\leq \sum_{j=0}^k  \sum_{i=1}^{j\neq k} C_{ij} \bigg( \int_{S^1} \Big\Vert  V^{(j-i+1)}(\psi_1(t)) \big( P(\psi_1'(t),...,\psi_1^{(i)}(t))  \\
&\;\;\;\;\;\;\;\;\;\;\;\;\;\;\;\;\;\;\;\;\;\;\;\;\;\;\;\;\;\;\;\;\;\;\;\;\;\;\;\;\;\;\;\;\;\;\;\;\;\;\;\;\;\;\;\; -  P(\psi_2'(t),...,\psi_2^{(i)}(t)) \big) \Big\Vert^2_g  dt \bigg)^{\frac{1}{2}} \\
&\:\:\:\:\:+ C_k \bigg( \int_{S^1} \Big\Vert  V'(\psi_1(t)) \big( P(\psi_1'(t),...,\psi_1^{(k)}(t))  -  P(\psi_2'(t),...,\psi_2^{(k)}(t)) \big) \Big\Vert^2_g  dt \bigg)^{\frac{1}{2}} \\
&+ \sum_{j=0}^k  \sum_{i=1}^{j\neq k} C_{ij} \bigg( \int_{S^1} \Big\Vert \big( V^{(i)}(\psi_1(t)) - V^{(i)}(\psi_2(t)) \big) P(\psi_2'(t),...,\psi_2^{(j-i+1)}(t)) \Big\Vert^2_g \: dt \bigg)^{\frac{1}{2}} \\
&\:\:\:\:\:+ C_k \bigg( \int_{S^1} \Big\Vert \big( V^{(k)}(\psi_1(t)) - V^{(k)}(\psi_2(t)) \big) P(\psi_2'(t)) \Big\Vert^2_g \: dt \bigg)^{\frac{1}{2}} \\
&\leq \sum_{j=0}^k  \sum_{i=1}^{j\neq k} C_{ij} \bigg( \int_{S^1} \Big\Vert  V^{(j-i+1)}(\psi_1(t)) \Big\Vert^2_g \delta^2 C_{i,j,c} dt \bigg)^{\frac{1}{2}} \\
&\:\:\:\:\:+ C_k \bigg( \int_{S^1} C_V \vert \psi_1^{(k)}(t)  -  \psi_2^{(k)}(t) \vert dt \bigg)^{\frac{1}{2}} \\
&+ \sum_{j=0}^k  \sum_{i=1}^{j\neq k} C_{ij} \bigg( \int_{S^1} \Big\Vert \big( V^{(i)}(\psi_1(t)) - V^{(i)}(\psi_2(t)) \big) \Big\Vert^2_g C_{i,j,c}\: dt \bigg)^{\frac{1}{2}} \\
&\:\:\:\:\:+ C_k C_c \bigg( \int_{S^1} \Big\Vert \big( V^{(k)}(\psi_1(t)) - V^{(k)}(\psi_2(t)) \big) \Big\Vert^2_g \: dt \bigg)^{\frac{1}{2}} \\
&\leq \sum_{j=0}^k  \sum_{i=1}^{j\neq k} C_{ij} C_{V,c} \delta C_{i,j,c} + C_k C_V \delta + \sum_{j=0}^k  \sum_{i=1}^{j\neq k} C_{ij} C_{i,j,c} \alpha_{V,i}(\delta) + C_k C_c (2 \delta_\gamma + \delta_{\gamma,\delta}) \\
& < \epsilon
\end{align*}
In these inequalities P are the appropriate polynomials and $ \alpha_{V,i}(\delta) \xrightarrow{\delta \rightarrow 0} 0$.

\end{proof}

As $W^{k+2,2}(S^1,M) \times \mathcal{D}^{k+2} \subset W^{k+2,2}(S^1,M) \times W^{k+2,2}(S^1,S^1)$, $k \in \mathbb{N}_0$ form an sc-structure and all $f$ are continuous, we can view $f$ as a $sc^0$ map.

\begin{prop} \label{f_sc_smooth}  ~\par
The map 
\begin{align*}
f: W^{2,2}(S^1,M) \times \mathcal{D}^{2} &\longrightarrow W^{2,2}(S^1,M) \\
(v,\psi) \:\: &\longmapsto \psi_\ast v := v \circ \psi
\end{align*}
as a map between sc-manifolds is sc-smooth.
\end{prop}

\begin{proof}
In this proof we will use the equivalent notions of sc-differentiability as outlined in the lemma 4.6 and proposition 4.10 in \cite{shift_map}. As these notions of sc-differentiability are defined on sc-Banach spaces and not sc-manifolds, we would technically need to look at the homeomorphisms between $W^{k,2}(S^1,M) \times \mathcal{D}^k$ and its tangent space. This will be omitted, and we will work with f directly, instead of looking at its equivalent on the tangent spaces.
Throughout this proof we will use lemma \ref{sc0}.

For $k \in \mathbb{N}_0$ set $F_{k} := W^{k+2,2}(S^1,M) \times W^{k+2,2}(S^1,S^1)$, an open neighborhood $U_k \subset W^{k+2,2}(S^1,M) \times \mathcal{D}^{k+2} \subset F_k$ and $H_{k} := W^{k+2,2}(S^1,M)$.
So we get $f: U_k \rightarrow H_k$ is $sc^0$.
Note that 
\[ T_{(v,\psi)}F_k = \Gamma^{k+2}(TM,v) \times \Gamma^{k+2}(TS^1,\psi) \cong \Gamma^{k+2}(TM,v) \times W^{k+2,2}(S^1,\mathbb{R}) \]

At $(v,\psi) \in F_1$ we evaluate the differential with $(V_1,\Psi_1) \in T_{(v,\psi)}F_0$:
\[ Df(v,\psi)(V_1,\Psi_1) = (v'\circ\psi)\Psi_1 + V_1 \circ \psi = \psi_\ast v' \cdot \Psi_1 + \psi_\ast V_1 \]

$Df(v,\psi)$ is continuous as a map from $F_k$ to $H_k$ for all $k \in \mathbb{N}_0$ as both summands are continuous in $\Psi$ and $V$ using lemma \ref{sc0} and the fact that the multiplication in these Sobolev spaces is continuous, see for example \cite{sobolev_multi}.
The map $Df: U_{k+1} \oplus T_{(v,\psi)}F_k \rightarrow T_{\psi_\ast v}H_k$  is also continuous as we have $v \in W^{k+3,2}(S^1,M)$ and therefore $v' \in \Gamma^{k+2}(TM,v)$.
Therefore $f$ is $sc^1$.

For $(v,\psi) \in U_2$ the second sc-differential at $(V_1,\Psi_1),(V_2,\Psi_2) \in T_{(v,\psi)}F_1 $ is
\[ 
D^2f(v,\psi)(V_1,\Psi_1)(V_2,\Psi_2) = \psi_\ast v''\cdot \Psi_1 \Psi_2 + \psi_\ast V_2' \cdot \Psi_1 + \psi_\ast V_1' \cdot \Psi_2
\]
Via induction one gets as the m-th sc-differential at $(v,\psi) \in F_m$
\[
D^mf(v,\psi): \bigoplus_{l=1}^m T_{(v,\psi)}F_{m-1} \rightarrow T_{\psi_\ast v}H_0
\]
Which when evaluated at $(V_1,\Psi_1),...,(V_m,\Psi_m) \in T_{(v,\psi)}F_{m-1}$ yields
\[
D^mf(v,\psi)\big((V_1,\Psi_1),...,(V_m,\Psi_m)\big) = \psi_\ast v^{(m)}\cdot \prod_{l=1}^m \Psi_l \: + \sum_{l=1}^m \psi_\ast V_l^{(m-1)} \prod_{j=1,j\neq l}^m \Psi_j
\]
By lemma \ref{sc0} and the continuity of multiplication this m-th sc-differential is continuous as a map
\[
D^mf: F_{m+k} \oplus \bigoplus_{l=1}^m T_{(v,\psi)}F_{m+k-1} \rightarrow T_{\psi_\ast v}H_k
\]
Hence using proposition 4.10 in \cite{shift_map} we get that $f$ is of class $sc^m$.
\end{proof}

Let $\mathcal{D}$ be the group $\mathcal{D}^2$ with the sc-structure created by $\mathcal{D}^k$.

\begin{cor}  ~\par
 $\mathcal{D}$ is a sc-Lie group.
\end{cor}

\begin{proof}
Choose $M = S^1$. Using proposition \ref{f_sc_smooth} we get that the group multiplication is sc-smooth. 
$\mathcal{D}$ has a sc-manifold structure, as $\mathcal{D}^k$ are open subsets of scales in a sc-manifold.
\end{proof}

\subsection{Embedded loops}  

We want to study the length functional on embedded loops, since most of the results for the flow by curvature, which we will see later, are only known for embedded curves. Also the reparametrization action on embedded loops is free.

Let us define the space of k-times weak differentiable embedded curves
\[
\mathcal{E}^k(M) := \{v \in W^{k,2}(S^1,M) \big\vert \: v\text{ is an embedding }\}
\]
Embedding in this case means a k-times weak differentiable diffeomorphism onto its image. For $k \geq 2$ these spaces are open subsets of $W^{k,2}(S^1,M)$. 
Hence we can view $\mathcal{E}^2(M)$ as an sc-manifold using the induced sc-structure $\mathcal{E}^2_k(M) = \mathcal{E}^2(M) \cap W^{k+2,2}(S^1,M) = \mathcal{E}^{k+2}(M)$.

Let us also define the space of continuous embedded curves.
\[
\mathcal{E}^0(M) := \{v \in C^0(S^1,M) \big\vert \: v\text{ is an embedding }\}
\]

By restriction we get the action of the reparametrization group $\mathcal{D}^k$ on $\mathcal{E}^k(M)$.
\begin{align*}
f: \mathcal{E}^k(M) \times \mathcal{D}^k &\longrightarrow \mathcal{E}^k(M) \\
(v,\psi) \:\: &\longmapsto \psi_\ast v := v \circ \psi
\end{align*}
While the group action of $\mathcal{D}^k$ on $W^{k,2}(S^1,M)$ is not free, as for example the constant curves are fixed points of the action, it is a free action on $\mathcal{E}^k(M)$, as the embeddings are diffeomorphisms from $S^1$ onto their images.

Viewing the group action of $\mathcal{D}^k$ on $\mathcal{E}^k(M)$ as an equivalence relation we get embedded curves without parametrization, and hence we can choose a representative with a parametrization freely, for example a parametrization with constant speed.
Therefore the spaces of unparametrized embedded curves can be defined as
\[
E^k(M) := \mathcal{E}^k(M) \big/ \mathcal{D}^k
\]

For M a 2-dimensional manifold we can endow $E(M) :=E^2(M)$ with the sc-structure $E^2_k(M) = E^{k+2}(M)$.

\begin{rmk} ~\par
It is unclear if such a quotient of a sc-manifold and a sc-Lie group is in general again a sc-manifold. Is it enough for the action of the sc-Lie group to be free and proper? 

This would be an interesting direction for further research.

In the case of $E(M)$ we can explicitly prove that it is a sc-manifold for two-dimensional manifolds M, but this proof has no direct generalization.
\end{rmk}

Given a two-dimensional orientable manifold M, $E(M) :=E^2(M)$ with the sc-structure $E^2_k(M) = E^{k+2}(M)$ forms an sc-manifold, using the following lemma. The idea of this lemma stems from a similar argument for immersed $C^2$-loops in \cite{Angenent}.

\begin{lemma}  ~\par
For a 2-dimensional orientable manifold M and $k \geq 2$ the space $E^k(M)$ is locally homeomorphic to $W^{k,2}(S^1,\mathbb{R})$.
\end{lemma}

\begin{proof}
For any $u \in E^k(M)$ choose a parametrization $u:S^1 \rightarrow M$. Choose an extension of $u$ to a local $W^{k,2}$-diffeomorphism $\sigma: S^1 \times (-\epsilon,\epsilon) \rightarrow M$.
For any sufficiently small $ \tau \in W^{k,2}(S^1,\mathbb{R})$ we thus get that
\[ u_\tau(t) = \sigma(t,\tau(t)) \]
is also an embedded $W^{k,2}$-curve, i.e. $[u_\tau] \in E^k(M)$. Let 
\[U_\epsilon = \{\tau \in W^{k,2}(S^1,\mathbb{R})\: \big\vert \sup_{t \in S^1} \vert \tau(t) \vert < \epsilon \} \]
For sufficiently small $\epsilon > 0$ we get that the map
\begin{align*}
\phi : U_\epsilon &\rightarrow \phi(U_\epsilon) \subset E^k(M) \\
\tau &\mapsto u_\tau
\end{align*}
is a homeomorphism.
\end{proof}

The only thing now left to check is, if the transition maps 
\[ \phi_2^{-1} \circ \phi_1 : U_{\epsilon_1} \cap \phi_1^{-1} \circ \phi_2(U_{\epsilon_2}) \rightarrow \phi_2^{-1} \circ \phi_1( U_{\epsilon_1}) \cap U_{\epsilon_2} \]
are $sc^k$ for $U_{\epsilon_i} \subset W^{k+2,2}(S^1,\mathbb{R})$. \\
Take a $\tau \in U_{\epsilon_1} \cap \phi_1^{-1} \circ \phi_2(U_{\epsilon_2})$. Thus we get that $u_{1,\tau} \in \phi_1(U_{\epsilon_1}) \cap \phi_2(U_{\epsilon_2})$.
Now we reparametrize $u_2 \in \phi_2(U_{\epsilon_2})$ in such a way that we get the same parametrization in the frist argument of the $\sigma_i$.
\[ u_{1,\tau}(t) = \sigma_1(t,\tau(t)) = \sigma_2(t,\tilde{\tau}(t)) = u_{2,\tilde{\tau}}(t) \]
and therefore $(t,\tilde{\tau}(t)) = \sigma_2^{-1} ( \sigma_1 (t, \tau(t)))$.
From this we get a homeomorphism
\begin{align*}
\varphi: U_{\epsilon_1} \cap \phi_1^{-1} \circ \phi_2(U_{\epsilon_2}) &\rightarrow \phi_2^{-1} \circ \phi_1( U_{\epsilon_1}) \cap U_{\epsilon_2} \\
\tau \:\: &\mapsto \:\: \tilde{\tau}
\end{align*}
This homeomorphism is of class $sc^k$ since the reparametrization is $sc^k$ (proposition \ref{f_sc_smooth}) and $\sigma_2^{-1} \circ \sigma_1$ is a k+2-times weak differentiable diffeomorphism.
By construction $\varphi = \phi_2^{-1} \circ \phi_1$.

\subsection{The length functional on embedded loops}  

\begin{lemma} \label{smoothL}  ~\par
The Length functional $L: E(M) \rightarrow \mathbb{R}$ is sc-smooth.
\end{lemma}

\begin{proof}
As the length functional $L: W^{k,2}(S^1,M) \rightarrow \mathbb{R}$ is continuous and the length functional is invariant under reparametrizations, we also get that $L: E^k(M) \rightarrow \mathbb{R}$ is continuous for $k \geq 2$. Hence $L: E(M) \rightarrow \mathbb{R}$ is $sc^0$.

At $v \in E^{k+3}(M)$, represented by a constant speed parametrized v, we can evaluate the differential with $V_1 \in T_vE^{k+2}(M)$ to get
\begin{align*}
DL(v)(V_1) &= \int_{S^1} \frac{1}{\Vert v'(t) \Vert} g(v'(t),V'_1(t)) dt \\
&= -\int_{S^1} \frac{1}{\Vert v'(t) \Vert} g(v''(t),V_1(t)) dt +\int_{S^1} \frac{1}{\Vert v'(t) \Vert} \frac{d}{dt}g(v'(t),V_1) dt \\
&= -\int_{S^1} \frac{1}{\Vert v'(t) \Vert} g(v''(t),V_1(t)) dt - \int_{S^1} \frac{d}{dt} \frac{1}{\Vert v'(t) \Vert} g(v'(t),V_1) dt \\
&= -\int_{S^1} \frac{1}{\Vert v'(t) \Vert} g(v''(t),V_1(t)) dt
\end{align*}
The differential $DL: E^{k+3}(M) \times T_vE^{k+2}(M) \rightarrow \mathbb{R}$ is continuous because $v \in W^{k+3}(S^1,M)$ and therefore $v''$ is continuous.
Hence L is $sc^1$. From the last term we can also see what the critical points of L are, those v where $v''(t) = 0$, that is, curves which are geodesics.

The second differential at $v \in E^{k+4}(M)$, evaluated at $V_1,V_2 \in T_vE^{k+3}(M)$ is
\begin{align*}
&D^2L(v)(V_1,V_2) \\
&= -\int_{S^1} \frac{1}{\Vert v'(t) \Vert^3} g(v'(t),V'_2(t)) g(v'(t),V_1'(t)) dt + \int_{S^1} \frac{1}{\Vert v'(t) \Vert} g(V_2'(t),V'_1(t))dt\\
&= \int_{S^1} \frac{1}{\Vert v'(t) \Vert^3} g(v'(t),V'_2(t)) g(v''(t),V_1(t)) dt + \int_{S^1} \frac{1}{\Vert v'(t) \Vert} g(V_2'(t),V'_1(t))dt\\
\end{align*}
If we look at the second differential at a critical point u, we thus get
\begin{align*}
D^2L(u)(V_1,V_2) = \int_{S^1} \frac{1}{\Vert u'(t) \Vert} g(V_2'(t),V'_1(t))dt\\
\end{align*}
All higher differentials $D^mL(v)(V_1,...,V_m)$ are sums of integrals over products of $ \displaystyle{\frac{1}{\Vert v'(t) \Vert^{2j+1}}}$, $g(v'(t),V'_j(t))$ and $g(V_i'(t),V'_j(t))$ for some $1 \leq i,j \leq m$.
Hence $D^m$ are continuous in v and $(V_1,...,V_m)$. Using proposition 4.10 in \cite{shift_map} we therefore get that L is of class $sc^m$.
\end{proof}

\begin{cor}  ~\par
The critical points of the length functional $L: E(M) \rightarrow \mathbb{R}$ are the embedded geodesics on M.
\end{cor}

\begin{proof}
As seen in the proof of lemma \ref{smoothL}, all critical points of the length functional are geodesics and all geodesics are critical points of L.

\end{proof}

The same statements and similar proofs also might work for the length functional on the space of immersed curves in M.

Now let us look at examples of the length functional on different manifolds.

\begin{exa} \label{S^2}  ~\par
If we take $M = S^2$ we get that the critical points of the length are the great circles on $S^2$. These great circles all have the same length $2\pi$ and form a two-dimensional critical manifold.
\end{exa}

\begin{exa} \label{T^2}  ~\par
Now let us take $T^2 = \mathbb{R}^2/\mathbb{Z}^2$, the flat torus. Here we have for each critical value of L at least two critical manifolds diffeomorphic to $S^1 = \mathbb{R}/\mathbb{Z}$. The geodesics of finite length on the flat torus are all lines of rational slope, where the shortest geodesics are those of slope 0 and $\infty$. All these geodesics are embedded and therefore critical points of the length functional. The critical points of L come in $S^1$-families, since for every parametrized embedded geodesic in $T^2$, there is a $T^2$-family of geodesics of the same length created by the action of $T^2$ on itself changing the start point of the parametrized geodesic. Because, as we view unparametrized loops, the movement of the start point along the geodesic produces the same geodesic, only a $S^1$-family remains.
\end{exa}

These first two examples are manifolds having Riemannian metrics with symmetries, and hence the critical points are not isolated, but rather in critical manifolds. Also all of these critical manifolds are not only the critical manifolds for the length on embedded curves, but also on immersed curves.
For Riemannian metrics such that there are no symmetries, we expect isolated critical points of L, as in the next examples. The length viewed on immersed curves in these examples would have significantly more critical points.

\begin{figure}[!h]
\centering
\includegraphics[width=0.8\linewidth]{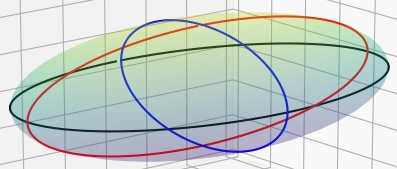}
\caption{\small The 3 embedded geodesics on an ellipsoid}
\label{3geodesics}
\end{figure}

\begin{exa} \label{Ellipsoid} ~\par
Let us now look at the ellipsoid $M = \big\{ (x,y,z) \in \mathbb{R}^3 \big\vert a x^2 + b y^2 + c z^2 = 1 \big\}$ for $a > b > c$. Using the theorem of the three geodesics (see \cite{curve_Grayson, 3geodesics}), we get that there are three embedded geodesics on such an ellipsoid, which can be seen in figure \ref{3geodesics}. As these three geodesics all have different lengths, we get that each critical value of the length L corresponds to one isolated critical point of L.

If we compare this to example \ref{S^2}, we see that changing the manifold or Riemannian metric even slightly might change the critical manifolds significantly. This is to be expected, as this behavior also exists, when one perturbs a Morse-Bott function to obtain a Morse function.
\end{exa}

\begin{exa} \label{skewed_torus}  ~\par
As the last example using a compact manifold, we will use the torus
$M = \big\{ (x,y,z) \in \mathbb{R}^3 \big\vert  (\sqrt{x^2 + y^2} - R)^2 + z^2 = (r + ax)^2 \big\}$ with $a \neq 0$ and $R > r > 0 $.

\begin{figure}[h]
\centering
\includegraphics[width=\linewidth]{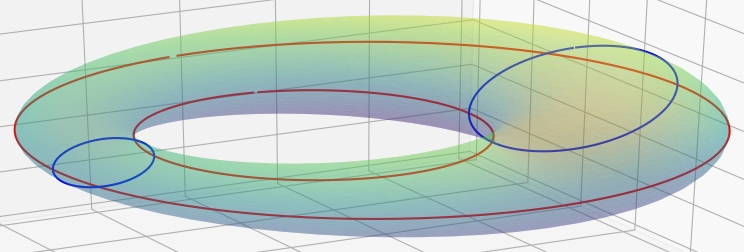}
\caption{\small The 4 simplest embedded geodesics on this more general torus}
\label{4geodesics}
\end{figure}

Because of the reparametrization action, we identify the homotopy class $(k,l) \in \mathbb{Z}^2 = \pi_1(T^2)$ with $(-k,-l)  \in \pi_1(T^2)$. On such a torus we have, at least generically, two embedded geodesics per identified homotopy class, except for the class of curves homotopic to a point, i.e. $(0,0) \in \pi_1(T^2)$, where there are no embedded geodesics. Thus compared to example \ref{T^2}, where each of these homotopy classes did have a $S^1$-family of embedded geodesics, we are now left with only two out of that $S^1$-family. 

For a generic $a \neq 0$, we get that all these embedded geodesics, and thus critical points of the length, have different lengths. (The generic condition is necessary, because there are $a \neq 0$, where embedded geodesics, that are not homotopic, have the same length.) 

\end{exa}

We thus have two examples of compact manifolds with isolated critical points, three on the ellipsoid and infinitely many on the torus.

And now, as a last example, we will look at a noncompact manifold, which is shaped roughly like a connected hyperboloid.

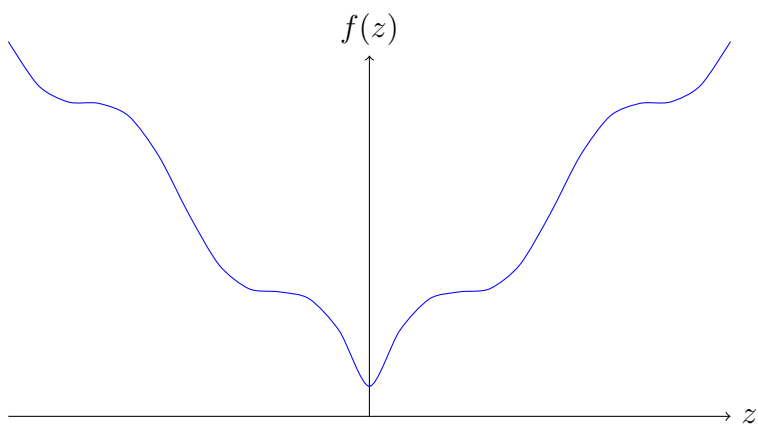
\begin{figure}[!h]
\centering
\begin{tikzpicture}
      \draw[->] (-4.8,0) -- (4.8,0) node[right] {$z$};
      \draw[->] (0,0) -- (0,4.8) node[above] {$f(z)$};
      \draw[scale=0.4,domain=-12:12,smooth,variable=\x,blue] plot ({\x},{1 + sqrt(\x*\x) + sin(180 / 3.141 * sqrt(\x*\x)) });
\end{tikzpicture}
\caption{\small $f(z) = \beta (z) z + \sin(\beta (z) z) + 1$}
\label{step_fct}
\end{figure}

\begin{exa} \label{stepped_cylinder} ~\par

Take a surface of revolution R of the graph of a function $f:\mathbb{R} \rightarrow (0,\infty)$, that is monotonely increasing for $z \geq 0$, monotonely decreasing for $z \leq 0$ and $\displaystyle{f(z) \xrightarrow{s \rightarrow \pm \infty} \infty}$. \\
Take for example the function $f(z) =  \beta (z) z + \sin(\beta (z) z) + 1$, depicted in figure \ref{step_fct}, with $\beta$ as in example \ref{exweights}.
This function has critical points $z_k = (2k+1)\pi$, $k \in \mathbb{Z}$ and $z^\ast = 0$.

 On $R = \big\{(x,y,z) \in \mathbb{R}^3 \: \big\vert \: x^2 + y^2 = (f(z))^2 \big\}$ the only embedded loops that are geodesics, are homotopic to $S^1 \times \{0\} \subset R$, namely the geodesics whose image is $\big(f(z_k) S^1\big) \times \{z_k\} \subset R$ for $z_k$ the critical points of f. (Here we use $S^1 \subset \mathbb{R}^2$ instead of $S^1 = \mathbb{R}/\mathbb{Z}$.)
On this surface we thus have infinitely many critical points of the length, that are all homotopic to each other.

This manifold will be interesting, since although it is not compact, the theorem of convergence of flow lines of embedded geodesics still holds on this manifold.
\end{exa}

\newpage

\section{Curvature and gradient flow}  
From now on (M,g) will be a 2-dimensional orientable Riemannian manifold which is convex at infinity.
These are the same conditions on the underlying manifold M as in \cite{curve_Grayson} and \cite{curve_Gage}, whose results we will use.

\subsection{Gradient flow of the length}  

A gradient flow line of a $sc^1$ function f is a path $v \in sC_{\omega_\infty}^\infty(\mathbb{R},E)$, for E an sc-manifold, satisfying the gradient flow equation
\begin{equation} \label{gfleq}
\partial_t v(t) + \nabla f (v(t)) = 0
\end{equation}
for any $t \in \mathbb{R}$. 
Let us denote $\displaystyle{\lim_{t \rightarrow \pm \infty} v(t) = v^\pm}$.

\begin{lemma} [gradient flow lines flow downwards] \label{down_flow} ~\par
Let $v \in sC^\infty(\mathbb{R},E)$ be a solution of the gradient flow equation (\ref{gfleq}) and let $t_1 < t_2$.\\
Then $f(v(t_1)) \geq f(v(t_2))$ and the equality $f(v(t_1)) = f(v(t_2))$ holds if and only if $v(t) = v_0 \in crit(f)$.
\end{lemma}

\begin{proof}
Using the gradient flow equation (\ref{gfleq}) and the definition of the gradient we can compute:
\begin{align*}
\frac{d}{dt} f(v(t)) = df(v(t)) \, \partial_tv(t) = - df(v(t)) \nabla f (v(t)) = - \Vert \nabla f (v(t)) \Vert^2_{g_{E}^{}} \leq 0 
\end{align*}
Here we get ''$= 0$'' if and only if $ \nabla f (v(t)) = 0$, and therefore if and only if $v(t) \in crit(f) \subset E$.
Hence we know that $f(v(t_1)) \geq f(v(t_2))$ and the equality $f(v(t_1)) = f(v(t_2))$ holds if and only if $v(t) \in crit(f)$.
\end{proof}

Obviously we will apply this to the length functional on embedded curves to get the gradient flow of the length.

In this thesis we will mostly ignore gradient flow lines, where $v^+$ or $v^-$ are not critical points of the length on embedded curves. 
Such gradient flow lines occur for example when $v^+$ is a constant curve, that is the gradient flow shrinks the embedded curves to a point.
Another example would be gradient flow lines where $\displaystyle{\lim_{t \rightarrow - \infty} L(v(t)) = \infty}$, so no limit curve $v^-$ exists.

We will thus only focus on gradient flow lines between critical points of the length functional, the embedded geodesics.


\begin{exa}
Both on $S^2$ and $T^2$ from examples \ref{S^2} and \ref{T^2}, the only gradient flow lines between critical points of L, are constant gradient flow lines, because all homotopic critical points are of the same length. Hence there are no nonconstant gradient flow lines of the length between the great circles on $S^2$ or the geodesics on $T^2$.
\end{exa}

\subsection{Flow by curvature}  

The evolution of a starting curve $u: S^1 \rightarrow M$ under the flow by curvature will be the family of smooth curves 
\begin{equation} \begin{aligned}
u: S^1 \times [0,T) &\longrightarrow M \\
(s,t) &\longmapsto u(s,t)
\end{aligned} \end{equation}
satisfying a heat equation
\begin{equation} \label{heateq} \begin{aligned}
\frac{\partial u}{\partial t} = k N
\end{aligned} \end{equation}
where k is the curvature of u and N is the Normal vector to u.
While in $\mathbb{R}^2$ the flow by curvature shrinks all embedded curves to points, see \cite{plane_curve_GH,plane_curve_Grayson}, on more general 2-dimensional manifolds it exhibits a much more interesting behavior. 

This curvature flow is also known as the curve shrinking flow, which is exemplified by the following equation
\begin{equation} \label{-k^2}
\frac{\partial}{\partial t}L(u(t,s)) = - \int_{S^1} k^2(u(t,s)) ds
\end{equation}

\begin{lemma} \cite[theorem 0.1]{curve_Grayson} \label{0.1} ~\par
Let $u : S^1 \rightarrow M$ be a smooth embedded curve with $u: S^1 \times [0, t_\infty ) \rightarrow M$ satisfying $ \frac{\partial u}{\partial t} = k N$, with $t_\infty \in \mathbb{R}$ being the largest time, such that the flow of u is well defined.

If $t_\infty$ is finite, then u converges to a point. If $t_\infty$ is infinite, the curvature k of u converges to 0 in the $C^\infty$-seminorms.
\end{lemma}

\begin{lemma} \cite[theorem 3.1]{curve_Gage} \label{3.1} ~\par
Let $u: S^1 \times [0, T] \rightarrow M$ be a solution of the evolution equation (\ref{heateq}) for some finite T. If $u(.,0)$ is an embedded curve and the curvature is bounded uniformly on $S^1 \times [0,T]$, then $u(.,t)$ are embedded for all $t \in [0,T]$.
\end{lemma}

If we start with an embedded curve u, whose $t_\infty = \infty$, as the curvature converges to zero, it has uniformly bounded curvature for all $T \in \mathbb{R}$. Hence the curvature flow of u only consists of embedded curves.

\begin{rmk}[equivalence of gradient flow of L and the curvature flow]

Using equations (\ref{heateq}) and (\ref{-k^2}) we get
\begin{align*}
- \int_{S^1} k^2(t,s) ds = \frac{\partial}{\partial t}L(v(t,s)) &= g^{}_{L^2(S^1, M)}\big(\nabla L (v(t,s)), \partial_t v(t,s)\big) \\
 &= \int_{S^1} g\big(\nabla L(t,s), k(t,s)N(t,s)\big) ds
\end{align*}
Since this equality holds for any v and any t we get
\[ k^2(t,s) = - g\big(\nabla L(t,s), k(t,s)N(t,s)\big) \]
Because $L(t,s) = L(v(t,s))$ is indepedent of $s \in S^1$, we know that $\nabla L(t,s)$ has no component tangential to v, and thus using the previous equation we get
\[ \nabla L(t,s) = - k(t,s)N(t,s) \]
Therefore we know that the gradient flow of the length functional and the flow by curvature are the same flow.
Hence we will use them interchangably.
\end{rmk}

Using lemmata \ref{0.1} and \ref{3.1}, we thus get that the gradient flow starting in an embedded curve u either ends in one point, or in a geodesic on M in the closure of E(M). Within E(M) all geodesic are critical points of the length. 

Thus, as long as a gradient flow line of the length remains within E(M), and thus $v^\pm \in E(M)$, the limit curves $v^\pm$ are embedded geodesics and therefore critical points of the length.

\begin{figure}[!h]
\centering
\includegraphics[width=0.69\linewidth]{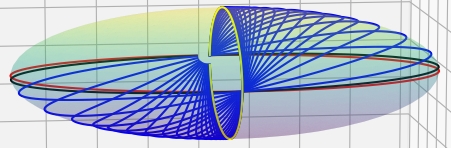}
\caption{\small The curvature flow (blue) on the ellipsoid, starting at the black curve close to the red geodesic, the shortest geodesic (yellow) as the limit of the flow}
\label{ellipsoid4}
\end{figure}

\begin{exa} \label{ellipsoid_flow}
Looking at the ellipsoid (example \ref{Ellipsoid}), where we have three embedded geodesics, most of the embedded curves will shrink to a point.
But if we look at curves bisecting the ellipsoid into two parts of equal area, analogously to theorem 5.1 in \cite{curve_Gage}, those curves will develop into one of the geodesics. One of these can be seen in figure \ref{ellipsoid4}. The curves as seen in figure \ref{ellipsoid4} form one gradient/curvature flow line from the longest embedded geodesic to the shortest geodesic.
\end{exa}

\begin{exa} \label{torus_flow}
Now let us look at the torus from example \ref{skewed_torus}. On this surface we have the topological restriction of the homotopy class of the initial curve.

All curves representing the homotopy class, which includes constant curves, i.e. the homotopy class $(0,0) \in \pi_1(T^2)$, will shrink to a constant curve, i.e. a point, under the curvature flow. There is no embedded geodesic in this homotopy class.

All curves in any other homotopy class will converge to one of the two embedded geodesics in that homotopy class. One example of such a flow can be seen in figure \ref{torus3}. In that figure we can see two curvature flows, which are part of two different gradient flow lines from the longer (red) geodesic to the shorter (yellow) geodesic.

\begin{figure}[h]
\centering
\includegraphics[width=0.9\linewidth]{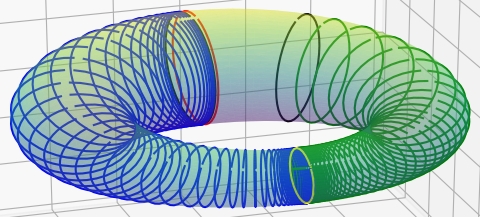}
\caption{\small Curvature flows (blue and green) on the torus, starting at the two black curves, both converging to the same yellow geodesic. Another perspective of this flow can be seen in figure \ref{torus4} in the appendix. }
\label{torus3}
\end{figure}
\end{exa}

\begin{exa} \label{cone_flow}
As a last example, let us look at the manifold R from example \ref{stepped_cylinder}. On this surface all embedded geodesics are homotopic. All embedded curves homotopic to these embedded geodesics will shrink to one of them under the curvature flow. Two such curvature flows can be seen in figure \ref{cone1}. All other embedded curves shrink to a point.

\begin{figure}[!h]
\centering
\includegraphics[width=0.8\linewidth]{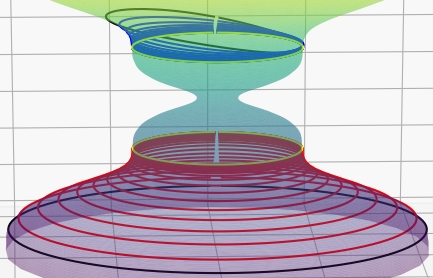}
\caption{\small Curvature flows (blue and red), starting at the two black curves, converging to the two yellow geodesics.}
\label{cone1}
\end{figure}
\end{exa}

\subsection{Convergence of flow lines}  

\begin{lemma} [Arzela-Ascoli] \label{arzela-ascoli} \cite[Theorem 2.5.14]{metric_geom} ~\par
In a compact metric space any sequence of curves with uniformly bounded lengths contains a uniformly converging subsequence.
\end{lemma}

The proof of this lemma in \cite{metric_geom} uses the ability to reparametrize curves to obtain curves with constant speed.
Hence this lemma works for the unparametrized curves in E(M), but it would not have worked with the parametrized embedded loops in $\mathcal{E}(M)$.

\begin{lemma} \label{M+-cpt}
Let M be a compact manifold, $v^-,v^+$ two embedded curves in M with $L(v^-) \geq L(v^+)$. Then the set 
\begin{equation}
\mathcal{M}^\pm_{C^0} := \big\{ v \in E^0(M) \: \big\vert \: L(v^-) \geq L(v) \geq L(v^+) \big\}
\end{equation}
has a compact closure $\overline{\mathcal{M}}^\pm_{C^0}$ in the $C^0(S^1,M)$ norm.
\end{lemma}

\begin{proof}
Take a sequence $(v_\nu) \subset \mathcal{M}^\pm_{C^0}$. Parametrize these $v_\nu(s)$ with constant speed, and since their length is bounded by $L(v^-)$, we can use the Arzela-Ascoli theorem (lemma \ref{arzela-ascoli}). We get that there exists a subsequence $v_{\nu_j}(s)$ that uniformly converges to a $v(s)$. This limit curve $v(s)$ still satisfies $L(v^-) \geq L(v) \geq L(v^+)$. 
We thus have $v_{\nu_j} \xrightarrow{C^0} v$ in $\overline{\mathcal{M}}^\pm_{C^0}$. Since any $v_\nu \in \overline{\mathcal{M}}^\pm_{C^0}$ can be approximated by $v_{\nu_i} \in \mathcal{M}^\pm_{C^0}$, we get that every $(v_\nu) \subset \overline{\mathcal{M}}^\pm_{C^0}$ has a convergent subsequence.
\end{proof}

\begin{lemma} \label{M+-cpt2}
Let M be a Riemannian manifold, $v^-,v^+$ two homotopic embedded curves in M with $L(v^-) \geq L(v^+)$. Let 
\begin{align*}
M^\pm = \big\{ p \in M \: \big\vert \: &\text{there exists a curve v homotopic to $v^-$ and an } s \in S^1 \\
&\text{ such that } v(s) = p \big\}
\end{align*}
be a subset of a compact set $K \subset M$. \\
Then the set 
\begin{equation}
\mathcal{M}^\pm_{C^0,h} = \big\{ v \in E^0(M) \: \big\vert \: L(v^-) \geq L(v) \geq L(v^+) \text{, v homotopic to v } \big\}
\end{equation}
has a compact closure $\overline{\mathcal{M}}^\pm_{C^0,h}$ in the $C^0(S^1,M)$ norm.
\end{lemma}


\begin{proof}
This proof works mostly the same as the proof of lemma \ref{M+-cpt}.
Take a sequence $(v_\nu) \subset \mathcal{M}^\pm_{C^0,h}$. We can use the Arzela-Ascoli theorem (lemma \ref{arzela-ascoli}) since we have curves in the compact metric space K. We get a subsequence $v_{\nu_j} \xrightarrow{C^0} v$ in $\overline{\mathcal{M}}^\pm_{C^0}$.
\end{proof}

If M is compact and thus satisfies lemma \ref{M+-cpt} it also satifies the conditions of lemma \ref{M+-cpt2} and hence $\overline{\mathcal{M}}^\pm_{C^0,h}$ is compact.

\begin{exa}
Let us use the surface of revolution R created from the graph of $f(z) = \beta (z) z + \sin(\beta (z) z) + 1$ as in example \ref{stepped_cylinder}. Let $v^-$ and $v^+$ be critical points of L on R, with $v^- = f(z_k)S^1 \times\{z_k\}$. Thus we know that $L(v^-) = 2\pi f(z_k)$. From $f(z) > f(z_k)$ for all $\vert z \vert > \vert z_k \vert$, we get that every embedded curve v with $L(v^-) \geq L(v)$ has to have at least one point in $R_{z_k} = \big\{ (x,y,z) \in R \: \big\vert \: z \in [-\vert z_k \vert,\vert z_k \vert \, ] \big\}$. We thus get that $v(s) \in R_{\vert z_k \vert+L(v^-)}$. All $R_z \subset R$ are compact.
Therefore this example satisfies the conditions of lemma \ref{M+-cpt2} and hence $\overline{\mathcal{M}}^\pm_{C^0,h}$ is compact.
\end{exa}

For the next theorem we need another quite general Arzela-Ascoli theorem.

\begin{lemma} \label{Arzela-Ascoli} [Arzela-Ascoli theorem] \cite[theorem 47.1]{topology} ~\par
Let $X$ be a space and let $(Y,d)$ be a metric space. Give $C(X,Y)$ the topology of compact convergence, let $\mathcal{F}$ be a subset of $C(X,Y)$.
\begin{enumerate}[a)]
\item If $\mathcal{F}$ is equicontinuous under d and the set
\[ \mathcal{F}_a = \{f(a) \: \vert \: f \in \mathcal{F} \} \]
has compact closure for each $a \in X$, then $\mathcal{F}$ is contained in a compact subspace of C(X,Y).
\item The converse holds if X is locally compact Hausdorff.
\end{enumerate}
\end{lemma}

\begin{thm} [Weak compactness] \label{weak_compactness} ~\par
Let M be compact or let the conditions of lemma \ref{M+-cpt2} be fulfilled. Let $L: E(M) \rightarrow \mathbb{R}$ be the length functional on embedded curves on M. Let there be a sequence $(v_\nu)_{\nu \in \mathbb{N}} \subset  sC_{\omega^{}_\infty}^\infty(\mathbb{R},E(M))$ of gradient flow lines of the length functional with 
\[ \lim_{t \rightarrow \pm \infty} v_\nu(t) = v^\pm\]
 where $v^-,v^+$ are critical points of L. \\
Then there exists a subsequence $(v_{\nu_j})$ and a gradient flow line v so that 
\[v_{\nu_j} \xrightarrow{sC^\infty_{loc}} v\]
\end{thm}

\begin{proof} ~\par
\textbf{Step 1:} $v_\nu(s)$ are equicontinuous. \\
Proof of step 1: The length of a path u from $t_1$ to $t_2$ in $C^0(S^1,M)$ is defined as 
\[ l(u) = \int_{t_1}^{t_2} \Vert \partial_t u \Vert^{}_{C^0(S^1,M)} dt \]
The metric on  $C^0(S^1,M)$ thus is
\[ d(u_1,u_2) = \inf \big\{ l(u) \: \big\vert \: u: [t_1,t_2] \rightarrow C^0(S^1,M), u(t_1) = u_1 \: \text{and } u(t_2) = u_2 \big\} \]
The gradient flow lines $v_\nu(s)$ are such paths in  $C^0(S^1,M)$. Using the gradient flow equation we get
\begin{align*}
&d\big(v_\nu(t_1),v_\nu(t_2)\big) \leq l \Big(v_\nu \Big\vert_{[t_1,t_2]} \Big) = \int_{t_1}^{t_2} \Vert \partial_t v_\nu (s) \Vert^{}_{C^0(S^1,M)} dt \\
&= \int_{t_1}^{t_2} \Vert \nabla L (v_\nu(s)) \Vert^{}_{C^0(S^1,M)} dt \leq \int_{t_1}^{t_2} c \: dt  = c (t_2- t_1)
\end{align*}
Here we have $\Vert \nabla L (v_\nu(s)) \Vert^{}_{C^0(S^1,M)} \leq c \,$ since $\nabla L$ at $v_\nu$ is a continuous function on the compact set $\overline{\mathcal{M}}^\pm_{C^0,h}$.
The compactness of this set is given by lemma \ref{M+-cpt} or \ref{M+-cpt2}. $\nabla L$ at $v_\nu$ is continuous since $v_\nu(s) \in E(M)$,  the embedding $E^2(M) \hookrightarrow E^0(M)$ is continuous and $\nabla L$ is continuous on $E^2(M)$.\\
Altogether $v_\nu(s)$ are equicontinuous.

\textbf{Step 2:} There exists a subsequence $v_{\nu_j}$ of the gradient flow lines $v_\nu$ and a gradient flow line $v$ such that $v_{\nu_j} \xrightarrow{C^0_{loc}} v$ in $C^0(\mathbb{R}, E^0(M))$.\\
Proof of step 2: To prove this step we want to use the Arzela-Ascoli theorem (lemma \ref{Arzela-Ascoli}). Take $X = \mathbb{R}$, $Y = E^0(M)$ and $\mathcal{F} = (v_\nu) \subset C^0(\mathbb{R}, E^0(M))$.
Using step 1 we get that $(v_\nu)$ is equicontinuous. The set $\mathcal{F}_s = ( v_\nu(s) )$ is relatively compact since it is a subset of a compact set $\overline{\mathcal{M}}^\pm_{C^0,h}$ as in lemma \ref{M+-cpt} or \ref{M+-cpt2}. \\
Using Arzela-Ascoli we thus get, that $(v_\nu)$ are part of a compact subset of  $C^0(\mathbb{R}, E^0(M))$ and thus $v_{\nu_j} \xrightarrow{C^0_{loc}} v$. \\
Since $\nabla L$ is continuous, we get $\nabla L(v_{\nu_j}(s)) \rightarrow \nabla L(v(s))$.

\textbf{Step 3:} The gradient flow line v lies in $C^0(\mathbb{R}, E^\infty(M))$. \\
Proof of step 3: Since v is a gradient flow line of L, it is also a flow line of the curvature flow and therefore, using a constant speed parametrization
\[ \partial_t v(t,s) = k(t,s)N(t,s) = \partial_s^2 v(s,t) \]
Because $v(s_0) \in E^0(M)$ we get that $v(s_0,t) \in L^2(S^1,M)$ for any $s_0 \in \mathbb{R}$. Using theorem 2.4 in \cite{parabolic_regularity} and the fact that embedded curves stay embedded, we get that $v(s,t) \in E^k(M)$ for any $s > s_0$ and $k \in \mathbb{N}$. 

\textbf{Step 4:} Bootstrapping to $sC_{\omega^{}_\infty}^\infty(\mathbb{R},E(M))$\\
We have previously shown
\[
\partial_t v_{\nu_j} = - \nabla L(v_{nu_j}) \xrightarrow{C^0_{loc}} - \nabla L(v) = \partial_t v
\]
and thus get
\[
v_{\nu_j}  \xrightarrow{C^1_{loc}}  v
\]
Be $s_0 \in \mathbb{R}$, choose local coordinates $U \subset E^\infty(M)$ around $v(s_0)$. For an $\epsilon > 0$ small enough and $j_0 \in \mathbb{N}$ big enough, we thus know that $v_{\nu_j}(s) \in U$ for all $s \in (s_0 - \epsilon,s_0 + \epsilon)$ and $j \geq j_0$. \\
Using the fact, that the length functional is sc-smooth on U (see lemma \ref{smoothL}), we know that $\nabla L$ is also sc-smooth and thus the chain rule and Leibniz rule can be applied to
\[
\partial_t v_{\nu_j} = - \nabla L (v_{\nu_j})
\]
This results in the equation
\[
\partial^k_t v_{\nu_j} = F_k(v_{\nu_j}, \partial_t v_{\nu_j}, .... ,\partial^{k-1}_t v_{\nu_j})
\]
where the $F_k$ are continuous for all $k \in \mathbb{N}$. \\
Via induction over k, using
\[
\partial^k_t v_{\nu_j} = F_k(v_{\nu_j}, \partial_t v_{\nu_j}, .... ,\partial^{k-1}_t v_{\nu_j}) \xrightarrow{C^0_{loc}} F_k(v, \partial_t v, .... ,\partial^{k-1}_t v) =  \partial^k_t v
\]
and
\[
v_{\nu_j} \xrightarrow{C^{k-1}_{loc}} v
\]
we get that
\[
v_{\nu_j} \xrightarrow{C^{k}_{loc}} v
\]
and we we thus get analogously to proposition \ref{scinfty}
\[
v_{\nu_j} \xrightarrow{sC^\infty_{loc}} v
\]

\end{proof}

\subsection{Breaking of flow lines}  

Let f be any functional, for example the length L.

\begin{defn} ~\par
Let $v^\pm \in crit(f)$. A \textbf{broken gradient flow line} from $v^-$ to $v^+$ is a tuple $u = (v^1, . . . , v^m)$ for some $m \in \mathbb{N}$ so that the following properties hold:
\begin{enumerate}[i)]
\item for all $k \in \{1,...,m\}$ the gradient flow line $v^k$ is nonconstant
\item $\displaystyle{\lim_{t \rightarrow - \infty} v^1(t) = v^-}$
\item $\displaystyle{\lim_{t \rightarrow \infty} v^k(t) = \lim_{t \rightarrow - \infty} v^{k+1}(t)}$ for all $k \in \{1,...,m-1\}$
\item $\displaystyle{\lim_{t \rightarrow  \infty} v^m(t) = v^+}$
\end{enumerate}
\end{defn}

The reparametrization of a gradient flow line $v$ by a $r \in \mathbb{R}$ is defined by $r_\ast v(t) = v(r+t)$.

\begin{defn} ~\par
Let $v_\nu \in sC^\infty(\mathbb{R},E)$ be a sequence of gradient flow lines, where there exist $v^\pm \in crit(f)$ so that $\displaystyle{\lim_{t \rightarrow \pm \infty} v_\nu(t) = v^\pm}$.
Let $u = (v^1, . . . , v^m)$ be a broken gradient flow line from $v^-$ to $v^+$.

Then the $v_\nu$ \textbf{Floer-Gromov converges} to $u$, if there exist a sequence $r^k_\nu \in \mathbb{R}$ for each $k \in \{1,...,m\}$ such that 
\[
(r_\nu^k)_\ast v_\nu \xrightarrow{sC^\infty_{loc}} v^k
\]
\end{defn}

Now we can talk about the convergence of gradient flow lines of the length towards broken gradient flow lines.

\begin{thm} [Breaking of gradient flow lines] ~\par
Let $v_\nu \in sC^\infty(\mathbb{R},E)$ be a sequence of gradient flow lines of the length, where there exist $v^\pm \in crit(L)$, so that $\displaystyle{\lim_{t \rightarrow \pm \infty} v_\nu(t) = v^\pm}$.
Let there be only finitely many critical values of L with value less than $L(v^-)$. Let all the critical points with length less than $L(v^-)$ be isolated critical points. \\
Then there exists a subsequence $v_{\nu_j}$ and a broken gradient flow line \\${u = (v^1, . . . , v^n)}$ from $v^-$ to $v^+$, so that 
\[
v_{\nu_j} \xrightarrow{Floer-Gromov} u
\]
\end{thm}

\begin{proof}
Proof by induction over $n \in \mathbb{N}$: \\
A(n): There exist a subsequence $v_{\nu_j}$, $u_n = (v^1, . . . , v^l)$ a broken gradient flow line for some $l \leq n$ and a sequence $r^k_j \in \mathbb{R}$ for each $k \in \{1,...,l\}$, such that
\begin{enumerate}[i)]
\item $\displaystyle{(r_j^k)_\ast v_{\nu_j} \xrightarrow{sC^\infty_{loc}} v^k}$ for all $k \in \{1,...,l\}$
\item $\displaystyle{\lim_{t \rightarrow - \infty} v^1(t) = v^-}$
\item If $l < n$, then $\displaystyle{\lim_{t \rightarrow  \infty} v^l(t) = v^+}$
\end{enumerate}

Proof of A(1): \\
Since the critical points are isolated, we can choose an open neighbourhood V of $v^-$, such that $\overline{V} \cap crit(L) = \{ v^- \}$ and V is contained in some $\epsilon^-$-ball around $v^-$. 
Define $r^1_\nu := \inf \{ t \in \mathbb{R} \vert v_\nu(t) \notin V \} < \infty$, the ''first exit time''. Using the weak compactness theorem (theorem \ref{weak_compactness}), we get a subsequence $v_{\nu_j}$ and a gradient flow line $v^1$ such that 
\[
(r_{\nu_j}^1)_\ast v_{\nu_j} \xrightarrow{sC^\infty_{loc}} v^1
\]
Since we have $v^1(0) = \lim_{j \rightarrow \infty} (r_{\nu_j}^1)_\ast v_{\nu_j}(0) = \lim_{j \rightarrow \infty}  v_{\nu_j}(r_{\nu_j}^1) \in \partial V$ and ${\partial V \cap crit(L) = \emptyset}$, we get that $v^1$ is not constant.

We therefore define $u^1 = (v^1)$ and $r_j^1 = r_{\nu_j}^1$.
By construction this fulfills properties i) and iii).
Since we have $\displaystyle{(r_{\nu_j}^1)_\ast v_{\nu_j} \xrightarrow{sC^\infty_{loc}} v^1}$ and $(r_{\nu_j}^1)_\ast v_{\nu_j}(t) \in \overline{V}$ for all $t \leq 0$, we get that $\displaystyle{\lim_{t \rightarrow -\infty} v^1(t) \in \overline{V}}$ for all $t \leq 0$. Because $v^-$ is the only critical point in $\overline{V}$ and V is bounded, we get that $\displaystyle{\lim_{t \rightarrow -\infty} v^1(t) = v^-}$ and thus property ii) is fulfilled.

\textbf{Induction step} $A(n) \Rightarrow A(n+1)$: \\
Take $u_n = (v^1, . . . ,v^l)$, a broken gradient flow line fulfilling A(n). \\
\textbf{Case 1} $\displaystyle{\lim_{t \rightarrow \infty} v^l(t) = v^+}$: \\

Set $u_{n+1} = u_n$ and thus $u_{n+1}$ fulfills A(n+1). \\
\textbf{Case 2} $\displaystyle{\lim_{t \rightarrow \infty} v^l(t) \neq v^+}$:

We hence have $l = n$. Also there exists a $\displaystyle{\lim_{t \rightarrow \infty} v^n(t) = (v^n)^+ \in crit(L)}$ such that $L((v^n)^+) \geq L(v^+)$, because $(v^n)^+$ is the limit of the flow by curvature and therefore either an embedded geodesic or a point. Since $\displaystyle{L((v^n)^+) = \lim_{t \rightarrow \infty} \lim_{j \rightarrow \infty} L\big( (r_{\nu_j}^n)_\ast v_{\nu_j} (t) \big) \geq \lim_{t \rightarrow \infty} \lim_{j \rightarrow \infty} L(v^+) = L(v^+) }$ using lemma \ref{down_flow}, it can not be a constant loop, and therefore is an embedded geodesic.

Take an open neighborhood W of $(v^n)^+$, so that $\overline{W} \cap crit(L) = \{ (v^n)^+ \}$ and W is contained in some $\epsilon_n$-ball around $(v^n)^+$. Since $\displaystyle{\lim_{t \rightarrow \infty} v^n(t) = (v^n)^+}$ there exists a $t_0 \in \mathbb{R}$ so that $v^n(t_0) \in W$. Using 
\[
(r_{\nu_j}^n)_\ast v_{\nu_j} \xrightarrow{sC^\infty_{loc}} v^n
\] 
we get that there exists $j_0 \in \mathbb{N}$, so that for all $j \geq j_0$ we have ${(r_{\nu_j}^n)_\ast v_{\nu_j}(t_0) \in W}$. Define 
$R_j = \inf \{ r \geq 0 \, \vert \, (r_{\nu_j}^n)_\ast v_{\nu_j}(t_0 + r) \notin V \}$. Then from 
\[
\lim_{t \rightarrow \infty} (r_{\nu_j}^n)_\ast v_{\nu_j}(t) = v^+ \notin W
\]
it follows that $R_j \leq \infty$. But since 
\[
(r_{\nu_j}^n)_\ast v_{\nu_j} \xrightarrow{sC^\infty_{loc}} v^n
\] 
we get that $\displaystyle{\lim_{j \rightarrow \infty} R_j = \infty}$. Let us now define $r_{\nu_j}^{n+1} := r_{\nu_j}^n + t_0 + R_j $.

From theorem \ref{weak_compactness} we get that there exists a further subsequence and a gradient flow line $v^{n+1}$, so that
\[
(r_{\nu_j}^{n+1})_\ast v_{\nu_j} \xrightarrow{sC^\infty_{loc}} v^{n+1}
\]
This gradient flow line $v^{n+1}$ is nonconstant, as 
\[
v^{n+1}(0) = (r_{\nu_j}^{n+1})_\ast v_{\nu_j}(0) = v_{\nu_j}(r_{\nu_j}^n + t_0 + R_j) = (r_{\nu_j}^n)_\ast v_{\nu_j} (t_0 + R_j) \in \partial W
\]
and $\partial W \cap crit(L) = \emptyset$.

\begin{claim}
\[
\lim_{t \rightarrow -\infty} v^{n+1}(t) = (v^n)^+ 
\]
\end{claim}
\noindent \textbf{Proof of the claim}: Argument by contradiction

Assume that
\[
\lim_{t \rightarrow -\infty} v^{n+1}(t) =: \hat{v} \neq (v^n)^+ 
\]
($\hat{v}$ is either an embedded geodesic, i.e. an element in crit(L) or it is not an element in E(M)) \\
\textbf{Case A} $L(\hat{v}) = \infty$ \\
Using lemma \ref{down_flow} we get $\displaystyle{L(\hat{v}) = \lim_{t \rightarrow \infty} \lim_{j \rightarrow \infty} L\big( (r_{\nu_j}^{n+1})_\ast v_{\nu_j} (t) \big) \leq L(v^-) }$ and thus have a contradiction \\
\textbf{Case B} $\hat{v} \in crit(L)$ or $\hat{v} \in \partial E(M)$ 

Choose $U \subset E(M)$ an open neighbourhood of $\hat{v}$ such that $W \cap U = \emptyset$.

From $\displaystyle{\lim_{t \rightarrow -\infty} v^{n+1}(t) = \hat{v} }$ and $\hat{v} \in U$ it follows that there exists a $t_1 < 0$, such that $v^{n+1}(t) \in U$ for all $t < t_1$.
Therefore there exists a $j_1$ such that for all $j \geq j_1$ we have $(r_{\nu_j}^{n+1})_\ast v_{\nu_j}(t_1) \in U$.

But we also know that $(r_{\nu_j}^n)_\ast v_{\nu_j}(t) \in W$ for all $t \in [t_0,t_0 + Rj]$. From this it follows through calculating
\[
(r_{\nu_j}^n)_\ast v_{\nu_j}(t) = v_{\nu_j}(t + r_{\nu_j}^n) = v_{\nu_j}(t + r_{\nu_j}^{n+1} - t_0 - R_j) = (r_{\nu_j}^{n+1})_\ast v_{\nu_j}(t - t_0 - R_j)
\]
that $(r_{\nu_j}^{n+1})_\ast v_{\nu_j}(t) \in W$ for all $t \in [-R_j,0]$.

As we have that $\displaystyle{\lim_{j \rightarrow \infty} R_j = \infty}$, we know that there exists a $j_2 \geq j_0,j_1$, such that $R_{j_2} < t_1 < 0$. 
Altogether we thus get $(r_{\nu_j}^{n+1})_\ast v_{\nu_j}(t_1) \in W \cap U$, which is a contradiction to $W \cap U = \emptyset$. This proves the claim. \qed

Setting $u_{n+1} = (v^1, . . . , v^n,v^{n+1})$ we get a broken gradient flow line satisfying A(n+1). Thus the induction is complete.

Because there are only finitely many critical values less than $L(v^-)$ and by lemma \ref{down_flow} the gradient flow lines flow strictly downwards for nonconstant gradient flow lines, any broken gradient flow line $u_{n} = (v^1, . . . , v^n)$, with $\displaystyle{\lim_{t \rightarrow -\infty} v^1 = v^-}$, can only have finitely many breaking points.
If we take a broken gradient flow line $u_{n} = (v^1, . . . , v^n)$ fulfilling A(m) for $m>n$, using the property iii) of A(n) yields $\displaystyle{\lim_{t \rightarrow \infty} v^n = v^+}$.
Combining the last two sentences we get that there exists a $n \in \mathbb{N}$, such that $\displaystyle{\lim_{t \rightarrow \infty} v^n = v^+}$.

\end{proof}

\newpage

\begin{appendices}
\section {Visualization of the curvature flow} 

Accompanying this thesis, there is a programming task, to code a implementation of the curvature flow in a python program, in order to be able to visualize the curvature flow.

As the resolution of the curves is only finite, there are some plots, where the curves are not quite smooth.

In the images there will be a small gap at the beginning/end of each curve. This gap does not influence the calculations, it only appears in the depiction of the flow.

First let us look at the curve shrinking in the plane. As outlined in \cite{plane_curve_GH,plane_curve_Grayson} the curve becomes more and more convex and shrinks to a point under the curvature flow. The curve as it shrinks becomes more and more circular.

\begin{figure}[h]
\centering
\includegraphics[width=0.75\linewidth]{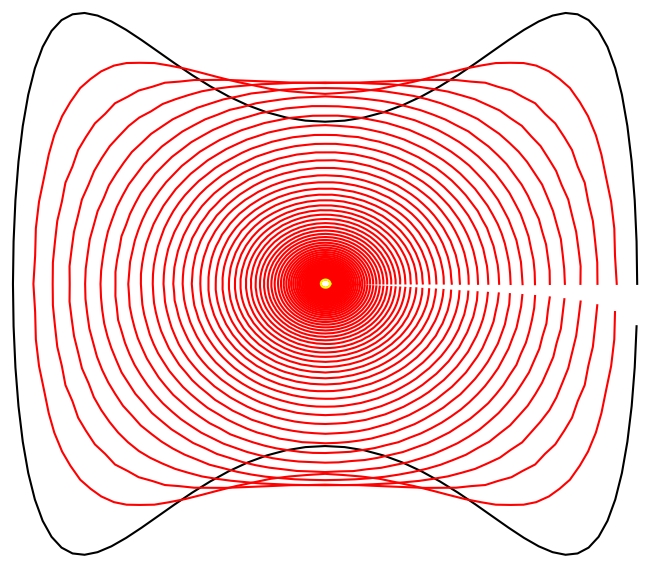}
\caption{\small From the initial curve (black) the curves under the flow first become convex and shrinks, becoming more and more circular. The curvature flow would fully shrink to a point, in the plot we end up with a small yellow near circular curve around that point.}
\label{convex0}
\end{figure}

\begin{figure}[h]
\centering
\begin{subfigure}{0.49\linewidth}
\includegraphics[width=\linewidth]{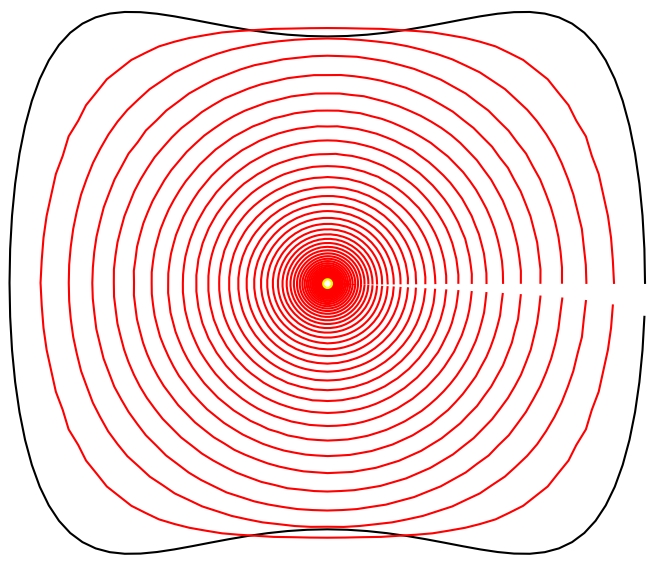}
\end{subfigure}
\begin{subfigure}{0.49\linewidth}
\includegraphics[width=\linewidth]{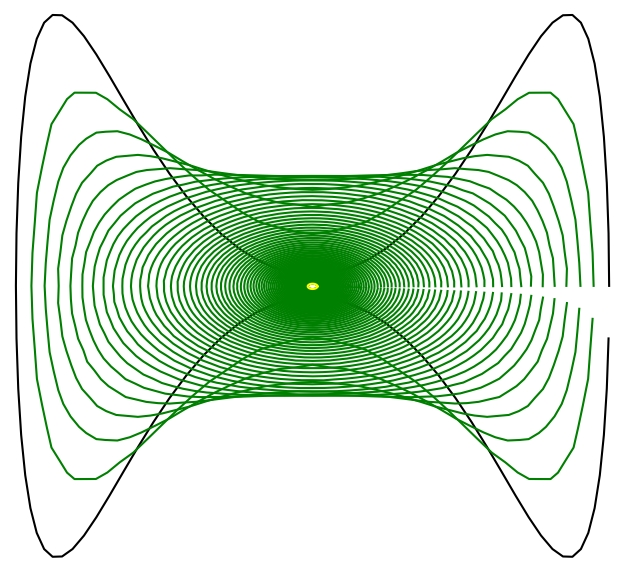}
\end{subfigure}
\caption{\small Two further examples of curve shrinking. On the left we see how fast the curves become circular. On the right we see how this works on a non-starshaped initial curve.}
\label{convex1}
\end{figure}

The curve shrinking can also applied to non-embedded curves, as can be seen in the images in figure \ref{convex2}.

\begin{figure}[!h]
\centering
\begin{subfigure}{0.49\linewidth}
\includegraphics[width=\linewidth]{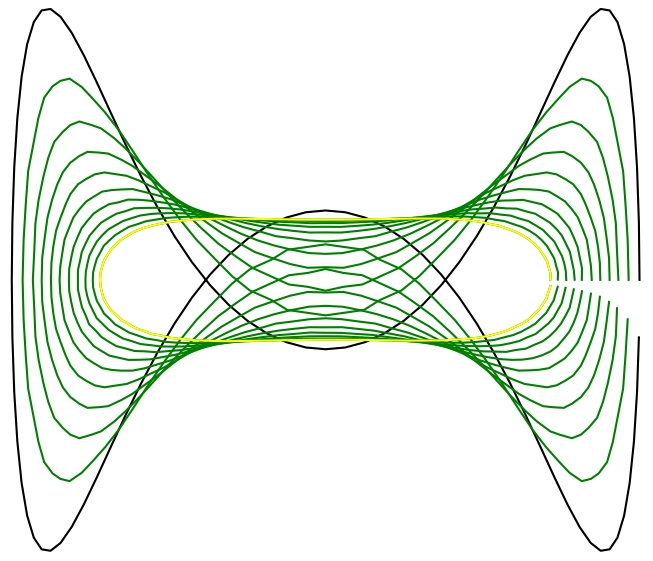}
\end{subfigure}
\begin{subfigure}{0.49\linewidth}
\includegraphics[width=\linewidth]{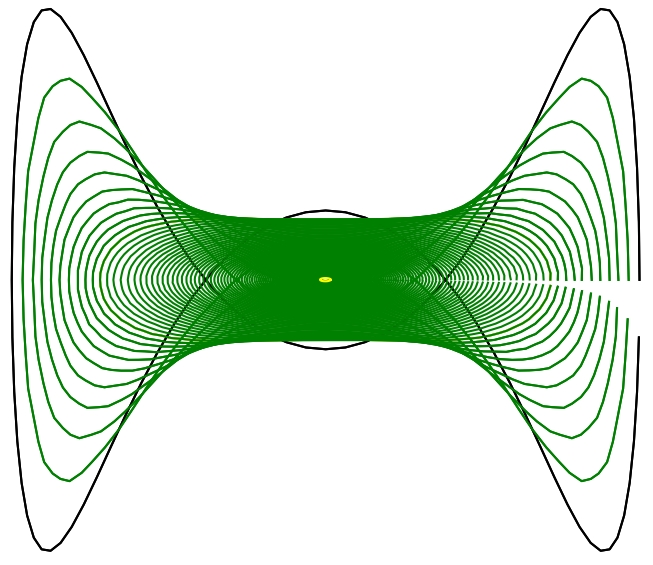}
\end{subfigure}
\caption{\small Curvature flow of a non-embedded curve, on the left becoming embedded and convex after a short time, on the right fully shrinking to a point}
\label{convex2}
\end{figure}

Now we will look at curves that are embedded in the manifolds, that have been looked at in previous examples.

As in example \ref{ellipsoid_flow} we will now look at an ellipsoid. First we look at a flow along a gradient flow line from the longest to the shortest embedded geodesic. One example of this is the plot on the title page.

\begin{figure}[h]
\centering
\includegraphics[width=0.7\linewidth]{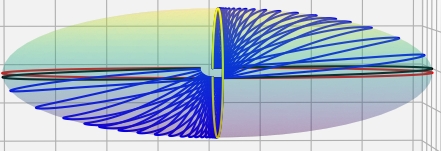}
\caption{\small Curvature flow starting at the black curve, like in figure \ref{ellipsoid4}, to the shortest embedded geodesic}
\label{Ellipsoid5}
\end{figure}

Next we look at the curvature flow from the longest to the second longest embedded geodesic.

\begin{figure}[!h]
\centering
\begin{subfigure}{0.6\linewidth}
\includegraphics[width=\linewidth]{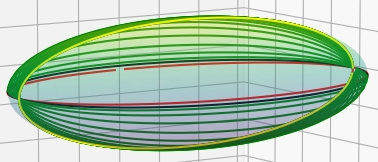}
\end{subfigure}
\begin{subfigure}{0.39\linewidth}
\includegraphics[width=\linewidth]{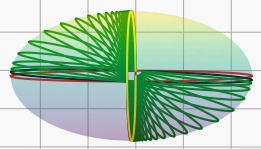}
\end{subfigure}
\caption{\small Curvature flow starting at the black curve near the longest geodesic (red), flowing to the second longest embedded geodesic (yellow)}
\label{Ellipsoid2}
\end{figure}

\begin{figure}[!h]
\centering
\includegraphics[width=0.75\linewidth]{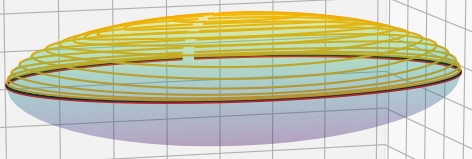}
\caption{\small Curvature flow starting at the black curve, just above a geodesic (red), flowing towards a single point at the top}
\label{Ellipsoid8}
\end{figure}

Ultimately most embedded curves on the ellipsoid shrink to a point, as probably only curves bisecting the ellipsoid into two parts of equal area shrink to one of the geodesics, as in theorem 5.1 in \cite{curve_Gage}.

As the next example, we will look at the torus from example \ref{skewed_torus}, where we already have one plot of the curvature flow in figure \ref{torus3}. 

\begin{figure}[h]
\centering
\includegraphics[width=0.62\linewidth]{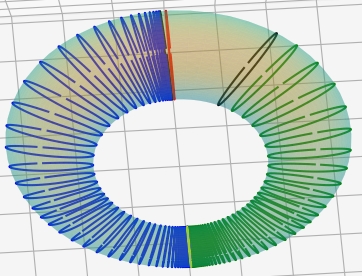}
\caption{\small Curvature flows (blue and green) starting at two curves (black), both converging to the same geodesic (yellow), same flow as in figure \ref{torus3}, but from a different perspective. The red curve is the longer geodesic.}
\label{torus4}
\end{figure}

Naturally we can also look at curves homotopic to another embedded geodesic, and the curvature flow of such curves.

\begin{figure}[!h]
\centering
\includegraphics[width=0.95\linewidth]{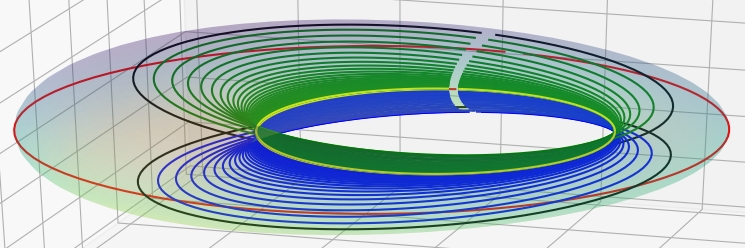}
\caption{\small Curvature flows (blue and green) starting at two curves (black), both converging to the same geodesic (yellow). The red curve is the longer geodesic.}
\label{torus5}
\end{figure}

We can also look at curves in the homotopy class $(1,1) \in \pi_1(T^2)$.

\begin{figure}[h]
\centering
\includegraphics[width=0.85\linewidth]{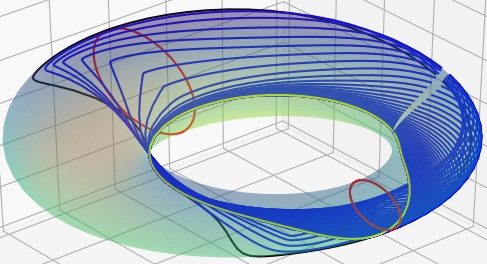}
\caption{\small Curvature flow (blue) starting at a curve (black), converging to the yellow geodesic. The red curves are geodesics for reference.}
\label{torus6}
\end{figure}

Now as a last set of images of the curvature flow, we look at the surface of revolution R from example \ref{stepped_cylinder}. 
We have already seen the curvature flow of two curves on this manifold in example \ref{cone_flow}. 
In the first figure we can see how a curve around a geodesic flows towards that geodesic.

\begin{figure}[!h]
\centering
\begin{subfigure}{0.6\linewidth}
\includegraphics[width=\linewidth]{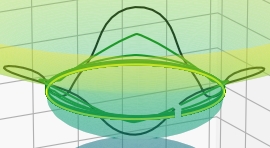}
\end{subfigure}
\begin{subfigure}{0.39\linewidth}
\includegraphics[width=\linewidth]{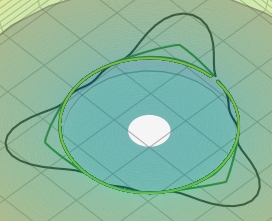}
\end{subfigure}
\caption{\small Curvature flow starting at the black curve, flowing to one of the second shortest embedded geodesics (yellow). The picture on the right is the same flow from above.}
\label{Ellipsoid2}
\end{figure}

In the next two figures we see curves flowing towards the shortest embedded geodesic.

\begin{figure}[!h]
\centering
\begin{subfigure}{0.9\linewidth}
\includegraphics[width=\linewidth]{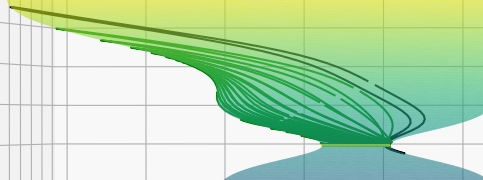}
\end{subfigure}
\begin{subfigure}{0.7\linewidth}
\includegraphics[width=\linewidth]{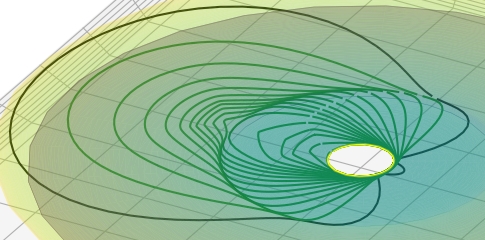}
\end{subfigure}
\caption{\small Curvature flow starting at the black curve, flowing to the shortest embedded geodesic (yellow). The upper picture is a side view of the flow, the lower picture depicts the flow from above.}
\label{Ellipsoid2}
\end{figure}

\begin{figure}[!h]
\centering
\begin{subfigure}{0.6\linewidth}
\includegraphics[width=\linewidth]{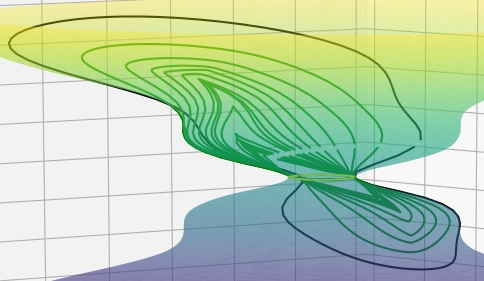}
\end{subfigure}
\begin{subfigure}{0.39\linewidth}
\includegraphics[width=\linewidth]{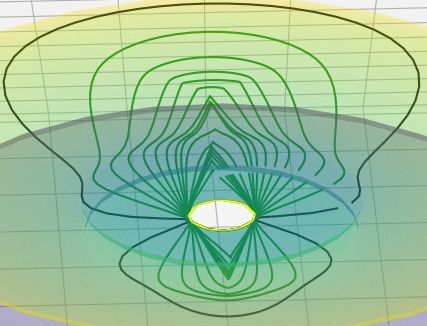}
\end{subfigure}
\caption{\small Curvature flow starting at the black curve, flowing to the shortest embedded geodesic (yellow). The picture on the right is the same flow from above.}
\label{Ellipsoid2}
\end{figure}

\newpage

Now as a last set of images, we look what happens, if we view the flow in backwards time. The curvature flow in backwards time is not well defined and thus we expect to get chaotic pictures.

\begin{figure}[!h]
\centering
\includegraphics[width=0.6\linewidth]{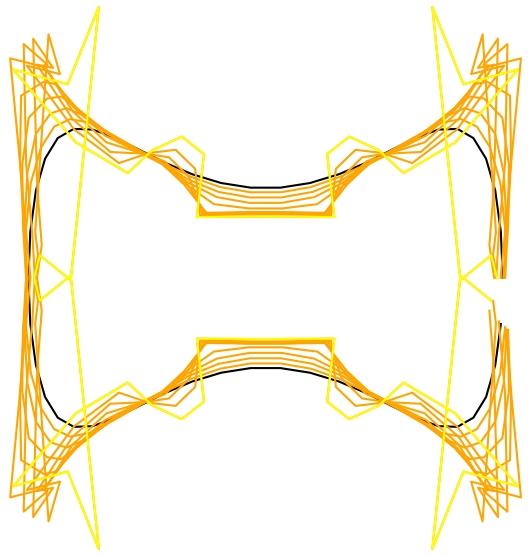}
\caption{\small Starting at the black curve, even after very few steps the flow gets chaotic.}
\label{torus5}
\end{figure}

\begin{figure}[!h]
\centering
\includegraphics[width=0.85\linewidth]{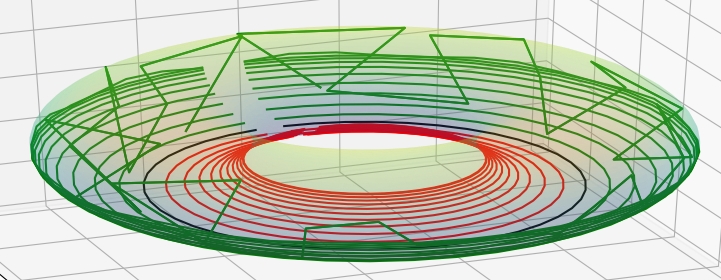}
\caption{\small Starting at the black curve, we look at the flow in forwards (red) and backwards time (green) after the same number of steps}
\label{torus5}
\end{figure}

In this last figure we see how for some time the backwards flow looks quite smooth. This is to be expected, as the inital curve lies on the gradient flow line from the outer to the inner geodesic on this torus. But small errors in the calculation compound and while in forward time these errors are diminished by the flow, in backwards time the errors get magnified.

\newpage

\end{appendices}

\newpage

\newpage

\bibliography{Flowbib}
\bibliographystyle{alpha}

\end{document}